\documentclass[12pt,a4paper]{amsart}
\usepackage{a4wide}
\usepackage{amsmath, amssymb, amsfonts,enumerate}
\usepackage[all]{xy}
\usepackage{amscd}
\usepackage{comment}
\usepackage{mathtools}

\newtheorem{theorem}{Theorem}[section]
\newtheorem{lemma}[theorem]{Lemma}
\newtheorem{corollary}[theorem]{Corollary}
\newtheorem{proposition}[theorem]{Proposition}

 \theoremstyle{definition}
 \newtheorem{definition}[theorem]{Definition}
 \newtheorem{remark}[theorem]{Remark}

 \newtheorem{example}[theorem]{Example}
\newtheorem{examples}[theorem]{Examples}

\numberwithin{equation}{section}
\newcommand {\N}{\mathbb{N}} 
\newcommand {\Z}{\mathbb{Z}} 
\newcommand {\R}{\mathbb{R}} 

\newcommand{\BB}{\mathcal{B}}
\newcommand{\CC}{\mathcal{C}}

\newcommand{\FF}{\mathcal{F}}

\newcommand{\PP}{\mathcal{P}}

\newcommand{\UU}{\mathcal{U}}

\newcommand{\PPf}{\mathcal{P}_{\text{fin}}}

\DeclareMathOperator{\sep}{sep}
\DeclareMathOperator{\spa}{span}
\DeclareMathOperator{\cov}{cov}

\DeclareMathOperator{\htop}{h_{\scriptstyle  \text{\rm top}}}

\DeclareMathOperator{\hsep}{h_{\scriptstyle \text{\rm sep}}}
\DeclareMathOperator{\hspa}{h_{\scriptstyle  \text{\rm spa}}}
\DeclareMathOperator{\hcov}{h_{\scriptstyle  \text{\rm cov}}}

\begin{document}
\title[Expansive actions with specification]{Expansive actions with specification on uniform spaces, topological entropy, and the Myhill property}
\author{Tullio Ceccherini-Silberstein}
\address{Dipartimento di Ingegneria, Universit\`a del Sannio, C.so
Garibaldi 107, 82100 Benevento, Italy}
\email{tullio.cs@sbai.uniroma1.it}
\author{Michel Coornaert}
\address{Universit\'e de Strasbourg, CNRS, IRMA UMR 7501, F-67000 Strasbourg, France}
\email{michel.coornaert@math.unistra.fr}
\subjclass[2010]{54E15, 37B05, 37B10, 37B15, 37B40, 43A07, 68Q80}
\keywords{uniform  space, uniformly continuous action, expansiveness, specification, amenable group, entropy, homoclinicity, strongly irreducible subshift, cellular automaton}
\begin{abstract}
We prove that every expansive  continuous action with the weak specification property
of an amenable group $G$ on a compact Hausdorff space $X$ has the Myhill property, i.e.,
every pre-injective continuous self-mapping of $X$  commuting with the action of $G$ on $X$  
is surjective.
This extends a result previously obtained by Hanfeng Li in the case when $X$ is metrizable.   
\end{abstract}
\date{\today}
\maketitle

\tableofcontents

\section{Introduction}

A \emph{topological dynamical system} is a pair $(X,G)$, where $X$ is a topological space and $G$ is a group  acting continuously on $X$.
Analogously,  a \emph{uniform dynamical system} is a pair $(X,G)$,
where $X$ is a uniform space
and $G$ is a group acting uniformly continuously on $X$.
Every uniform dynamical system may be regarded as a topological dynamical system.
Indeed, every uniform space has an underlying topology and every uniformly continuous self-mapping 
of a uniform space is continuous with respect to this topology.
In the other direction, every topological dynamical system $(X,G)$ with $X$ compact Hausdorff can be regarded as a uniform dynamical system.
Indeed, a compact Hausdorff space admits a unique uniform structure compatible with its topology and  every continuous self-mapping of the space is uniformly continuous with respect to this uniform structure. 
\par
An \emph{endomorphism} of a topological dynamical system $(X,G)$ is a continuous $G$-equivariant self-mapping
of $X$,  that is, a continuous map $f \colon X \to X$ such that 
$f(g x) = g f(x)$ for all $g \in G$ and $x \in X$.
One says that the topological dynamical system $(X,G)$ is \emph{surjunctive} if every injective endomorphism of 
$(X,G)$ is surjective.
The term ``surjunctive" was created by Gottschalk~\cite{gottschalk} in the early 1970s
and the search for  conditions  guaranteeing surjunctivity of certain classes of  dynamical systems has attracted much interest, especially in the last two decades  since the seminal work of Gromov~\cite{gromov-esav}.
  \par
Let $(X,G)$ be a uniform dynamical system.
Two points $x,y \in X$ are called \emph{homoclinic} if their  orbits are asymptotically close, i.e., for every entourage $U$ of $X$, there exists a finite subset $\Omega \subset G$ such that 
$(g x, g y) \in U$ for all $g \in G \setminus  \Omega$.
Homoclinicity is an equivalence relation on $X$.
An endomorphism of $(X,G)$  is called \emph{pre-injective} if its restriction to each homoclinicity class is injective.
\par
One says that a uniform dynamical system $(X,G)$ satisfies the \emph{Myhill property} if every 
pre-injective endomorphism of 
$(X,G)$ is surjective.
As injectivity  clearly implies pre-injectivity,
a uniform dynamical system  is surjunctive whenever it has the Myhill property.
\par
One says that a uniform dynamical system $(X,G)$ is \emph{expansive} if there is an entourage $U_0$ of $X$ such that
there is no pair $(x,y) \in X \times X$ with $x \not= y$ satisfying  $(g x, g y) \in U_0$ for all $g \in G$.
Given any set $A$, the full $G$-shift with alphabet $A$,
i.e., the system $(A^G,G)$, where $A^G$ is equipped with its uniform prodiscrete structure and $G$ acts on $A^G$ by the $G$-shift 
(see Formula~\eqref{e:shift-action}),
yields an example of an expansive uniform dynamical system.
\par
One says that the uniform dynamical system $(X,G)$ has the \emph{weak specification property} if
for every entourage $U$ of $X$, there is a finite subset $\Lambda = \Lambda(U) \subset G$
satisfying the following condition:
for every finite sequence $\Omega_1, \dots, \Omega_n$ of finite subsets of $G$ such that
$\Omega_j \cap \Lambda \Omega_k = \varnothing$ for all distinct $j,k \in \{1,\dots,n\}$ and every sequence $x_1,\dots,x_n$ of points of $X$,
there exists a point $x \in X$
such that $(g x, g x_i) \in U$ for all $g \in \Omega_i$ and $1 \leq i \leq n$.
\par
The main goal of the present paper is to  establish the following result.

\begin{theorem} 
\label{t:myhill}
Let $X$ be a compact Hausdorff space equipped with a continuous action of an amenable group $G$.
Suppose that  $(X,G)$ is expansive and has the weak specification property.
Then  $(X,G)$ has the Myhill property.
In particular,  $(X,G)$ is surjunctive.
 \end{theorem}

Theorem~\ref{t:myhill} has been previously established by Li \cite[Theorem~1.1]{li-goe} 
under the additional hypotheses that $X$ is metrizable and $G$ is countable.
By virtue of the Bryant-Eisenberg theorem,
 if a compact Hausdorff space $X$ admits an expansive continuous action of a countable group $G$,
then $X$ is necessarily metrizable
(see Theorem~\ref{t:expansive-metrizable} below).
Therefore Theorem~\ref{t:myhill}
reduces to Theorem~1.1  in~\cite{li-goe}  when  the group $G$ is countable. 
However, for any uncountable   group $G$,
there exist expansive topological dynamical systems with the weak specification property $(X,G)$,
with $X$  compact Hausdorff but   non-metrizable
(e.g.~the full shift $(A^G,G)$
with  $A$ a finite set with more than one element).
 \par
We cannot drop the weak specification hypothesis in Theorem~\ref{t:myhill}.
Indeed, consider a discrete space with  two distinct points  $X = \{x_1,x_2\}$
 and a   group $G$ fixing each point of $X$.
Observe that the space $X$ is compact metrizable and $(X,G)$ is expansive.
Moreover, each homoclinicity class is reduced to a single point so that 
 the endomorphism  $f \colon X \to X$ given by $f(x_1) = f(x_2) = x_1$ is  pre-injective.
As $f$  is not surjective,
we deduce   that $(X,G)$ fails to have the Myhill property.
Note however that this dynamical system  is surjunctive since $X$ is finite.
A non-surjunctive  example is provided
by the subshift  $(X,\Z)$, where $X \subset \{0,1\}^{\Z}$ consists of all bi-infinite sequences of $0$s and $1$s  containing at most
one  chain of $1$s.
Here also  $X$ is compact metrizable and $(X,\Z)$ is expansive.
However, the map $f \colon X \to X$ which replaces  the word $1 0$ if it appears in a configuration 
$x \in X$
by the word $11$
 is an injective endomorphism of $(X,\Z)$.
This endomorphism  is not surjective since
any configuration with exactly one occurrence of the symbol $1$ 
cannot be in the image of $f$. Therefore $(X,\Z)$ is not surjunctive.
\par
The expansiveness assumption cannot be either dropped
in Theorem~\ref{t:myhill}.  
Indeed, take any  topological dynamical system
$(X,G)$ with the weak specification property and $X$ compact Hausdorff with more than one point
(e.g. the $G$-shift on $\{0,1\}^G$).
Consider the topological dynamical system $(X^\N,G)$,
where $X^\N= \prod_{n \in \N} X$ is equipped with the product topology and  $G$ acts diagonally on $X^\N$. 
Then $X^\N$ is compact Hausdorff and the topological dynamical system
$(X^\N, G)$ has the weak specification property (cf.\ Proposition~\ref{p:WSP-produit}).
However,  $(X^\N,G)$ is not surjunctive
(and hence does not have the Myhill property)
since the map $f \colon X^\N \to X^\N$ defined by $f(x)(0) = x(0)$ and $f(x)(n) = x(n-1)$ for all 
$n \geq 1$,
is an injective but not surjective endomorphism of $(X^\N,G)$.
\par  
It follows from a result of Bartholdi~\cite{bartholdi-duality}
that, for any non-amenable group $G$,
there exists a finite set $A$ such that the full $G$-shift $(A^G,G)$ does not have  the Myhill property.
This shows in particular that   the amenability assumption cannot be removed from Theorem~\ref{t:myhill}.
 \par  
A subshift $X \subset A^G$ has the weak specification property for the $G$-shift if and only if $X$ is strongly irreducible
 (see Proposition~\ref{p:wsp-strong-irred} below).
Therefore, Theorem~\ref{t:myhill} yields the following.

\begin{corollary}
\label{c:strongly-myhill}
Let $A$ be a finite set and let $G$ be an amenable group.
Let $X \subset A^G$ be a strongly irreducible subshift.
Equip $X$ with the shift action of $G$ and the topology induced by the prodiscrete topology on $A^G$.
Then  $(X,G)$ has the Myhill property.
In particular, $(X,G)$ is surjunctive.
\qed
\end{corollary}
\par
In the case when $X$ is the full shift $A^G$ and $G = \Z^d$ (the free abelian group of finite rank $d$), Corollary~\ref{c:strongly-myhill} is due to  Myhill in~\cite{myhill}.
Myhill's result  was subsequently extended to  full shifts over finitely generated amenable groups
by Mach{\`{\i}},  Scarabotti,and the first author in~\cite{ceccherini}
and then to  strongly irreducible subshifts of finite type over finitely generated amenable groups
 by Fiorenzi in~\cite{fiorenzi-strongly}.
Finally, Corollary~\ref{c:strongly-myhill} was established in its full generality by the authors 
in~\cite{csc-myhill-monatsh}.
\par
The paper is organized as follows.
In Section~\ref{sec:background-material}, we introduce notation and review some background material
on uniform spaces, ultrauniform spaces, group actions, shift spaces, and amenable groups.
Section~\ref{sec:expansiveness} investigates expansiveness for uniform dynamical systems.
We prove the Bryant-Eisenberg theorem stating that any compact Hausdorff space admitting an  expansive continuous action of a countable group is necessarily metrizable (see Theorem~\ref{t:expansive-metrizable}).
Section~\ref{sec:top-entropy} 
is devoted to the definition and basic properties of the topological entropy $\htop(X,G)$ 
of a topological dynamical system $(X,G)$ with $G$ amenable.
In Section~\ref{sec:homoclinicity}, we introduce the notion of homoclinicity for 
uniform dynamical systems. We prove that two configurations in a shift space are homoclinic if and only if they coincide outside of a finite subset of the underlying group
(Proposition~\ref{p:homoclinicity-almost-equality}).
 In Section~\ref{sec:weak-specification}, we study weak  specification 
 for uniform dynamical systems.
Our definition of weak specification coincides with the one  in~\cite{chung-li} and~\cite{li-goe}
in the compact metrizable case.
Weak specification for subshifts with discrete alphabet is equivalent to strong irreducibility
(see Proposition~\ref{p:wsp-strong-irred}).
We also show that  if $X$ is a compact Hausdorff space with more than one point equipped with a continuous action with the weak specification property of an amenable group $G$,
then $\htop(X,G) > 0$ (see Theorem~\ref{t:entropie-positive}). 
The proof of Theorem~\ref{t:myhill} is given in Section~\ref{sec:proof-main-result} and is divided into two parts, which are of independent interest. 
We first show (Theorem~\ref{t:entropie-sousespace-propre}) that if $X$ is a compact Hausdorff space equipped with a continuous action with the weak specification property of an amenable group $G$ and $Y \subsetneqq X$ is a proper closed $G$-invariant subset  such that the action of $G$ restricted to  $Y$ is expansive, then
  $\htop(Y,G) < \htop(X,G)$. 
Next we establish   Theorem~\ref{t:entropie-diminue-failure-preinjectivite} which says that if $X$ and $Y$ are compact Hausdorff spaces equipped with expansive continuous actions of an amenable group $G$ such that the action of $G$ on $X$ has the weak specification property and $\htop(Y,G) < \htop(X,G)$, then any continuous $G$-equivariant map $f \colon X \to Y$ fails 
  to be  pre-injective. These two results combined together immediately yield Theorem~\ref{t:myhill}.\\
	
	\noindent
	{\bf Acknowledgments.} We express our deepest gratitude to the anonymous referee for the most careful reading of our manuscript and the numerous suggestions and remarks.

\section{Background material}
\label{sec:background-material}

\subsection{General notation}
Given sets $A$ and $B$, we write $A \subset B$ if every element in $A$ is also in $B$.
We denote by $A^B$ the set consisting of all maps $x \colon B \to A$.
When $A$ is a finite set, we write $|A|$ its cardinality.

\subsection{Subsets of $X \times X$}
 Let $X$ be a set.
 \par 
 We denote by $\Delta_X$ the \emph{diagonal} of $X$, that is,
 the subset $\Delta_X \subset X \times X$ consisting of all pairs $(x,x)$ with $x \in X$.
 \par
  Let $U \subset X \times X$.
Given a point $x \in X$, we denote by $U[x]$ the subset of $X$ consisting of all points $y \in X$ such that $(x,y) \in U$.
\par
One says  that $U$ is \emph{reflexive} if $\Delta_X \subset U$.
 \par
The \emph{inverse}  of $U$ is the subset  $U^* \subset X \times X$ 
consisting of all pairs $(x,y)$ such that 
$(y,x) \in U$.
One says that  $U$ is \emph{symmetric} if $U^* = U$.
\par
The \emph{composite}  of $U$ with another subset  $ V \subset X \times X$  is
the subset $U \circ V \subset X \times X$ consisting of all pairs
$(x,y) \in X \times X$ such that
 there exists   $z \in X$ with  $(x,z) \in V$ and $ (z,y) \in U$.
One says that  $U$ is \emph{transitive} if $U \circ U \subset U$.
\par
Observe that the set consisting of all subsets of $X \times X$ is a monoid for $\circ$ with identity element $\Delta_X$ and that  the map $U \mapsto U^*$ is an anti-involution of this monoid.
This  monoid is ordered for inclusion in the sense that
$U \circ V \subset U' \circ V'$ whenever $U,V,U',V' \subset X \times X$ satisfy $U \subset U'$ and $V \subset V'$.
This implies in particular that
$U \subset U \circ V$ if $\Delta_X \subset V$ and that
$V \subset U \circ V$ if $\Delta_X \subset U$. 
\par
Note also that a relation $\sim$ on $X$ is reflexive (resp.~symmetric, resp.~transitive)
 if and only if
its graph $\Gamma(\sim) \coloneqq \{(x,y) : x \sim y \}$ is a reflexive
(resp.~symmetric, resp.~transitive) subset of $X \times X$.

\subsection{Uniform spaces}
The theory of uniform spaces we briefly review here was introduced by Andr\'e Weil 
in~\cite{weil-uniforme}.
 The reader is referred 
to the monographs~\cite[Ch.~2]{bourbaki-top-gen}, \cite{james}, and \cite[Ch. 6]{kelley}
for a more detailed exposition.
\par
Let $X$ be a set. A \emph{uniform structure} on $X$ is a nonempty set $\UU$ of subsets of $X \times X$, whose elements are called the \emph{entourages} of $X$,  satisfying the following conditions:
\begin{enumerate}[(UN{I}-1)]
\item every  $U \in \UU$ is reflexive;
\item if $U \in \UU$ and $U \subset V \subset X \times X$, then $V \in \UU$;
\item if $U \in \UU$ and $V \in \UU$, then $U \cap V \in \UU$;
\item if $U \in \UU$, then $U^* \in \UU$;
\item if $U \in \UU$, then there exists $V \in \UU$ such that 
$V \circ V \subset U$.
\end{enumerate}

\begin{remark}
\label{rem:sym-entourage}
It immediately follows from conditions (UNI-3), (UNI-4), and (UNI-5) that, given any entourage $U \in \UU$, there exists a symmetric entourage $V \in \UU$ such that $V \circ V \subset U$.
\end{remark}Let $X$ be a set and let   $U \subset X \times X$. 

A set equipped with an uniform structure  is called a \emph{uniform space}.
\par
If $X$ is a uniform space and $Y \subset X$, 
then the uniform structure on $X$ naturally induces a uniform structure on $Y$.
The entourages of the uniform structure induced by $X$ on $Y$ are the sets of the form 
$U \cap (Y \times Y)$, where $U$ runs over all entourages of $X$.
\par
A subset $\BB$ of a uniform structure  $\UU$ on a set $X$ is called a \emph{base of entourages}  if for every entourage 
$U \in \UU$,
there exists an entourage $V \in \BB$ such that $V \subset U$.
\par
The \emph{discrete uniform structure} on a set $X$ is the uniform structure on $X$ whose entourages are all
the reflexive subsets of $X \times X$.
\par
If $X$ is a uniform space,
there is an induced topology on $X$ characterized by the fact that the neighborhoods of an arbitrary point $x \in X$ consist of the sets $U[x]$, where $U$ runs over all entourages of $X$.
This topology is Hausdorff if and only if the intersection of all the entourages of $X$ is reduced to the diagonal $\Delta_X $.
\par
One says that a topological space $X$ is \emph{uniformizable} if there is a uniform structure on $X$ inducing its topology.
\par
Every compact Hausdorff space $X$ is uniquely uniformizable, that is,
there is a unique uniform structure on $X$ inducing its topology. The entourages of this uniform structure are the neighborhoods of the diagonal in 
$X \times X$ (see \cite[Th\'eor\`eme~1, TG II.27]{bourbaki-top-gen}).
\par
When $X$ is a compact Hausdorff space,
it immediately follows from the normality of $X \times X$ that
the  closed entourages
(i.e., the entourages that are closed in $X \times X$)  
form a base of entourages
of the uniform structure on $X$.
\par
 If $d$ is a metric on a set $X$, then $d$ defines   a  uniform structure on $X$.
A base of  entourages for this uniform structure 
consists of all the symmetric sets
\begin{equation}
\label{e:Delta-d-epsilon}
\Delta_\varepsilon(X,d) \coloneqq 
\{(x,y) \in X \times X:   d(x,y) < \varepsilon\},
\end{equation}
with $\varepsilon > 0$.
The topology associated with this uniform structure coincides with the topology defined by the metric $d$.
\par
A uniform space $X$ is called \emph{metrizable} if there exists a metric $d$ on $X$ which
defines  the uniform structure on $X$.
The \emph{metrization theorem for uniform spaces} 
(see, e.g., \cite[Chapter~6, Theorem~13]{kelley}), states that a uniform space is metrizable
if and only if it is Hausdorff and admits a countable base of entourages.
\par
 A map $f \colon X \to Y$ between uniform spaces is said to be \emph{uniformly continuous} if
$(f \times f)^{-1}(W)$ is an entourage of $X$ for every entourage $W$ of $Y$.
One says that $f \colon X \to Y$ is a \emph{uniform isomorphism} if $f$ is bijective with  $f$ and $f^{-1}$ both uniformly continuous. 
\par
If $(X_k)_{k \in K}$ is a family of uniform spaces,
indexed by a set $K$, 
the \emph{product uniform structure} on the product set $X \coloneqq \prod_{k \in K} X_k$ is the smallest
(with respect to inclusion)
uniform structure on $X$ such that all projection maps $X \to X_k$, $k \in K$, are uniformly continuous.
In the case when the uniform structure on each $X_k$ is the discrete uniform structure, the product uniform structure on $X$ is called the
\emph{prodiscrete uniform structure}.

\subsection{Ultrauniform spaces}
A uniform structure $\UU$ on a set $X$ is called a \emph{ultrauniform structure} if $\UU$ admits a base of entourages consisting of graphs of equivalence relations.
In other words, $\UU$ is a ultrauniform structure if for every entourage $U \in \UU$, there exists a symmetric and transitive entourage $V \in \UU$
such that $V \subset U$. 
A set equipped with a ultrauniform structure is called a \emph{ultrauniform space}.

\begin{examples}
1) A set equipped with its discrete uniform structure is a ultrauniform space.
\par
2) The product of a family of ultrauniform spaces, equipped with its product uniform structure,  
is a ultrauniform space.
In particular, the prodiscrete uniform structure on a product of sets  is ultrauniform.
\par
3) If $X$ is a ultrauniform space and $Y \subset X$, then the uniform structure induced on $Y$ by  the uniform structure on $X$ is also ultrauniform.
\par
4) 
Let $(X,d)$ be a metric space and suppose that $d$ satisfies the \emph{ultrametric inequality}, 
i.e.,
$d(x,y) \leq \max(d(x,z),d(y,z))$ for all $x,y,z \in X$.
Then the uniform structure defined by $d$ on $X$ is ultrauniform.
Indeed, if $d$ satisfies the ultrametric inequality, then the sets $\Delta_\varepsilon(X,d)$ 
defined in~\eqref{e:Delta-d-epsilon} are clearly transitive.   
\end{examples}

\subsection{Actions}
An \emph{action} of a group $G$ on a set $X$ is a map $(g,x) \mapsto g x$ from $G \times X$ into $X$ such that
$g_1(g_2 x) = (g_1 g_2) x$ and $1_G x = x$ for all $g_1,g_2 \in G$ and $x \in X$.
\par
Let $X$ be a set equipped with an action of a group $G$.
Let $Y \subset X$.
For $g \in G$, we write $g Y \coloneqq \{ g y : y \in Y\}$.
Given a subset  $E \subset G$, we write
\begin{equation}
\label{e:def-cap-g-inv-Y}
Y^{(E)} \coloneqq \bigcap_{g \in E}  g^{-1} Y.
\end{equation}
Thus $x \in Y^{(E)}$ if and only if $g x \in  Y$ for all $g \in E$.
One says that  $Y$ is \emph{$G$-invariant} if $g Y =Y$ for all $g \in G$.
\par
Suppose that $(X_k)_{k \in K}$ is a family of sets and that a group $G$ acts on each of the sets $X_k$, $k \in K$.
Then $G$ naturally acts  on the product set $ \prod_{k \in K} X_k$
via the \emph{diagonal action} defined by
\[
g x \coloneqq (g x_k)_{k \in K}
\]
for all $x = (x_k)_{k \in K} \in \prod_{k \in K} X_k$.
\par
In particular, if a group $G$ acts on a set $X$, then $G$ naturally acts on $X \times X$ via the diagonal action.
\par
If a group $G$ acts on two sets $X$ and $X'$, one says that a map 
$f \colon X \to X'$ is \emph{$G$-equivariant} if $f(g x) = g f(x)$ for all $g \in G$ and $x \in X$.
\par
An action of a group $G$ on a topological space $X$ is said to be \emph{continuous} if the  map $x \mapsto g x$ is continuous on $X$ for every $g \in G$.
This amounts to saying that $g^{-1}U$ is an open subset of $X$ for every open subset $U \subset X$ and any $g \in G$.
\par
An action of a group $G$ on a uniform space $X$ is called \emph{uniformly continuous} if the  map 
$x \mapsto gx$ is uniformly continuous on $X$ for each $g \in G$.
This amounts to saying that $g^{-1} U$ is an entourage of $X$ for every entourage $U \subset X \times X$ of $X$ and any $g \in G$. 
\par
Let $X$ and $Y$ be topological spaces equipped with a continuous action of a group $G$.
One says that the topological dynamical system $(Y,G)$ is a \emph{topological factor} of the topological dynamical system $(X,G)$ if there exists 
a continuous $G$-equivariant surjective map from $X$ onto $Y$.
One says that the topological dynamical systems $(X,G)$ and $(Y,G)$ are
\emph{topologically conjugate} if there exists a $G$-equivariant homeomorphism between $X$ and $Y$.
Similarly, if $X$ and $Y$ are uniform spaces equipped with a uniformly continuous  action of a group $G$,
one says that the uniform dynamical system $(Y,G)$ is a \emph{uniform factor} of the uniform dynamical system $(X,G)$
if there exists a uniformly continuous $G$-equivariant surjective map from $X$ onto  $Y$.
One says that the uniform dynamical systems $(X,G)$ and $(Y,G)$ are \emph{uniformly conjugate} if there exists a $G$-equivariant uniform isomorphism between $X$ and $Y$.

\subsection{Covers}
Let $X$ be a set.
A \emph{cover} of $X$ is a set of subsets of $X$ whose union is $X$.
Let $\alpha$ and $\beta$ be  covers of $X$. 
The \emph{join} of $\alpha$ and $\beta$ is the cover $\alpha \vee \beta$  of $X$ consisting of all the  subsets of the form 
$A \cap B$ with $A \in \alpha$ and $B \in \beta$.
One says that $\beta$ is a \emph{subcover} of $\alpha$ if $\beta \subset \alpha$.
One says that  $\beta$ is a \emph{refinement} of $\alpha$ if
for every $B \in \beta$, there exists $A \in \alpha$ such that $B \subset A$.
Note that  $\alpha \vee \beta$ is a refinement of both $\alpha$ and $\beta$.
\par
When the set $X$ is equipped with an action of a group $G$, given a cover $\alpha$ of $X$ and an element $g \in G$, we denote by  $g \alpha$ the cover of $X$ consisting of all the subsets of the form $g A$ with $A \in \alpha$.
\par
When $X$ is a topological space, an \emph{open cover} of $X$ is a cover of $X$ whose elements are all open subsets of $X$.  

\subsection{Shift spaces}
Let $G$ be a group and let $A$ be a set called the \emph{alphabet}.
The set $A^G$ consisting of all maps $x \colon G \to A$
is called the set of \emph{configurations} over the group $G$ and the alphabet $A$.
\par
Given a subset $\Omega \subset G$ and a configuration $x \in A^G$, we shall denote by $x\vert_\Omega$ the restriction of $x$ to $\Omega$, i.e., the map 
$x\vert_\Omega \in A^\Omega$ given by
$x|_\Omega(g) = x(g)$ for all $g \in \Omega$.
\par
The \emph{shift} on $A^G$ is the action of $G$ on $A^G$  defined by
\begin{equation}
\label{e:shift-action}
gx(h) \coloneqq  x(g^{-1}h) \text{   for all $g,h \in G$ and $x \in A^G$.}
\end{equation}
\par
Suppose now that $A$ is a uniform space.
We equip $A^G = \prod_{g \in G} A$ with the product uniform structure.
A base of entourages for this uniform structure consists of all the sets 
\begin{equation}
\label{e:base-entourages-unif}
W(\Omega,U) \coloneqq \{(x,y) \in A^G \times A^G : (x(g),  y(g)) \in U \text{ for all } g \in \Omega  \},
\end{equation}
where $\Omega$ runs over all finite subsets of $G$ and $U$ runs over all entourages of $A$.
The shift action on $A^G$ is clearly uniformly continuous.
\par
A $G$-invariant closed subset $X \subset A^G$ is called a \emph{subshift} of $A^G$.
\par
In the case when $A$ is equipped with the discrete uniform structure,  the
corresponding product uniform structure on $A^G$ is the prodiscrete uniform structure and a base of entourages
of $A^G$ is formed by the sets
\begin{equation}
\label{e:base-entourages}
W(\Omega) \coloneqq \{(x,y) \in A^G \times A^G : x\vert_\Omega = y\vert_\Omega\},
\end{equation}
where $\Omega$ runs over all finite subsets of $G$ (see \cite{book}).
\par
Note that if $G$ is uncountable and $A$ is a Hausdorff (e.g.~discrete) uniform space  with more than one element,
then $A^G$ is not metrizable (not even first countable).

\subsection{Amenable groups}
There are many equivalent definitions of amenability for groups  
in the literature
(see for example~\cite{greenleaf}, \cite{paterson}, \cite{book}).
In the present paper, we shall only use 
the following one.

\begin{definition}
A group $G$ is called \emph{amenable} if there exists a net $\FF = (F_j)_{j \in J}$ of nonempty finite subsets $F_j \subset G$
such that 
\begin{equation}
\label{e:def-left-Folner-net}
\lim_{j \in J} \frac{|g F_j \setminus F_j|}{|F_j|} = 0 \quad \text{for all $g \in G$}.
\end{equation}
Such a net $\FF$ is called a \emph{left F\o lner net} for $G$.
\end{definition}

The class of amenable groups includes all finite groups, all abelian groups, all solvable groups, and all finitely generated groups with subexponential growth. 
Moreover, it is closed under taking subgroups, quotients, extensions, and directed limits.
\par
Given a group $G$, we denote by $\PPf(G)$ the set of all finite subsets of $G$.
A map $h \colon \PPf(G) \to \R$ is said to be \emph{subadditive} 
if $h(A \cup B) \leq h(A) + h(B)$
for all $A,B \in \PPf(G)$.
It is said to be \emph{right-invariant} 
if $h(A g) = h(A)$ for all $A \in \PPf(G)$ and $g \in G$.
The following convergence result is due to Ornstein and Weiss \cite{ornstein-weiss}
(see also~\cite[Section 1.3.1]{gromov-tids},   \cite{krieger}, \cite{csck}).

\begin{theorem}[Ornstein-Weiss lemma]
\label{t:OW}
Let $G$ be an amenable group and 
let $h \colon \PPf(G) \to \R$ be a subadditive  right-invariant map.
Then there exists a real number $\lambda \geq 0$, depending only on the map $h$,
with the following property:
if $ (F_j)_{j \in J}$ is a left F\o lner net  for $G$
then the net of real numbers 
$\left(\dfrac{h(F_j)}{|F_j|}\right)_{j \in J}$
converges to $\lambda$. 
\end{theorem}

\section{Expansiveness}
\label{sec:expansiveness}

An action of a group $G$ on a uniform space $X$ 
is called  \emph{expansive} if there exists  an entourage $U_0$ of $X$
satisfying the following property: 
for all points 
$x, y \in X$ with $x \not= y$,
  there exists an element $g \in G$ such that $(gx,gy) \notin U_0$.
Such an entourage $U_0$ is then called an \emph{expansiveness entourage} for the dynamical system 
$(X,G)$.

\begin{example}
\label{ex:shift-expansive}
Let $G$ be a group and let $A$ be a set.
Equip $A^G$ with its prodiscrete uniform structure and the shift action of $G$.
Then $(A^G,G)$ is expansive.
Indeed, the set $W(\{1_G\})$ defined by \eqref{e:base-entourages} is clearly an expansiveness entourage for the system. 
\end{example}

\begin{example}
\label{ex:sub-expansive}
Let $(X,G)$ be an expansive uniform dynamical system and let $Y \subset G$ be a $G$-invariant subset.
Then the uniform dynamical system $(Y,G)$ is also expansive.
Indeed, if $U_0$ is an expansiveness entourage for $(X,G)$, then $U_0 \cap (Y \times Y)$ is an expansiveness entourage for $(Y,G)$. 
\end{example}

\begin{example}
\label{ex:subshifts-expansive}
Combining  the two previous examples, we deduce that
if $G$ is a group, $A$ a set, and $X \subset A^G$ a subshift, then $(X,G)$ is expansive.
\end{example}
Expansiveness for uniformly continuous actions on uniform spaces has been investigated  in~\cite{bryant}, \cite{good-macias}  for 
iterates of uniformly continuous maps  and 
in~\cite{eisenberg-expansive}, \cite{lam-expansive},  \cite{csc-expansive-uni}
for actions of general groups.
\par
Observe that if a uniform space $X$ admits a uniformly continuous and expansive action of a group $G$ then the topology on $X$ is necessarily Hausdorff.
Indeed, if $U_0$ is an expansiveness entourage then  the diagonal in $X \times X$ is the intersection of the  entourages
$g^{-1}  U_0$, $g \in G$. 
\par
The following result was first established by Bryant~\cite[Theorem~1]{bryant} for 
$G = \Z$ and then extended to general countable groups by 
Eisenberg~\cite[Theorem~1]{eisenberg-expansive}
(see also~\cite[Corollary~2.8]{lam-expansive}).

\begin{theorem}
\label{t:expansive-metrizable}
Let $X$ be a compact Hausdorff space equipped with an expansive continuous  action of a countable  group $G$. Then $X$ is metrizable.
\end{theorem}

For the proof, we shall use the following result.

\begin{lemma}
\label{l:expansive-base-entourages}
Let $X$ be a compact Hausdorff space equipped with an expansive continuous  action of a   group $G$
and let $U_0$ be a closed  expansiveness entourage for $(X,G)$. 
Then, for every entourage $U$ of $X$, there exists a finite subset 
$E = E(U) \subset G$ such that
$\bigcap_{g \in E} g^{-1}U_0  \subset U$.
\end{lemma}

\begin{proof}
Let $U$ be an entourage of $X$ and let $V$ be an open neighborhood of the diagonal in $X \times X$ such that $V \subset U$.
Since $U_0$ is a closed expansiveness entourage,
the open sets $(X \times X) \setminus g^{-1} U_0$, $g \in G$, cover $(X \times X) \setminus V$.
By compactness of  $(X \times X) \setminus V$,  there exists a finite subset $E   \subset G$ such that
\[
(X \times X) \setminus V \subset \bigcup_{g \in E} ((X \times X) \setminus g^{-1} U_0).
\]
This implies $\bigcap_{g \in E} g^{-1} U_0 \subset V \subset U$.
 \end{proof}

\begin{proof}[Proof of Theorem~\ref{t:expansive-metrizable}]
By the metrization theorem for uniform spaces mentioned above, it suffices to show that $X$ admits a countable base of entourages.
Let $U_0$ be a closed expansiveness entourage.
By Lemma~\ref{l:expansive-base-entourages},  the set $\BB$ consisting of all the entourages of the form 
 $\bigcap_{g \in E} g^{-1} U_0$, where $E$ runs over all finite subsets of $G$, is a base of entourages of $X$.
As the group $G$ is countable, the set of its finite subsets is also countable.
Thus  $\BB$ is countable as well.
This shows that    $X$ is metrizable.
\end{proof}

\section{Topological entropy}
\label{sec:top-entropy}

\subsection{Topological entropy}
Topological entropy for  continuous self-mappings of topological spaces was first introduced by Adler, Konheim and McAndrew~\cite{adler}.
Their definition was directly inspired by the one given by Kolmogorov for measure-theoretic entropy.
\par
Let $X$ be a  topological space.
\par 
If $\alpha$ is a finite open cover of $X$, we denote by $N_X(\alpha)$, or simply $N(\alpha)$ if there is no risk of confusion on the ambient space, the minimal cardinality of a subcover of $\alpha$.
Observe that if $\alpha$ and $\beta$ are finite open covers of $X$, then  $\alpha \vee \beta$ is also a finite open cover of $X$. Moreover, it satisfies
\begin{equation}
\label{e:card-join-oc}
N(\alpha \vee \beta) \leq N(\alpha) \cdot N(\beta).
\end{equation}
Note also that if $\beta$ is a refinement of $\alpha$, then $N(\beta) \geq N(\alpha)$. 
\par
Suppose now that $X$ is equipped with a continuous action of a  group $G$. 
If $\alpha$ is a finite open cover of $X$ and $g \in G$, then $g \alpha$ is also 
a finite open cover of $X$ and
\begin{equation}
\label{e:card-g-alpha}
N(g \alpha) = N(\alpha).
\end{equation}
If $\alpha$ is an open cover of $X$ and $F$ is a finite subset of $G$, we define the open cover 
$\alpha^{(F)}$ by 
\[
\alpha^{(F)} \coloneqq \bigvee_{g \in F} g^{-1}  \alpha.
\]

\begin{lemma}
\label{l:h-foc-subadd}
Let $\alpha$ be a finite open cover of $X$.
Then the map $h \colon \PP_{\small fin}(G) \to \R$ defined by 
$h(F) \coloneqq  \log N(\alpha^{(F)})$
is subadditive and right $G$-invariant.
\end{lemma}

\begin{proof}
If $F$ and $F'$ are finite subsets of $G$, then the covers
$\alpha^{(F \cup F')}$ and $\alpha^{(F)}  \vee \alpha^{(F')}$ refine each other so that
\[
h(F \cup F') = \log N(\alpha^{(F \cup F')} ) 
= \log N(\alpha^{(F)} \vee \alpha^{(F')} ) \leq \log N(\alpha^{(F)}) + \log N(\alpha^{(F')} ) = h(F) + h(F')
\]
by applying~\eqref{e:card-join-oc}. This shows that $h$ is subadditive.
\par
On the other hand, for every $g \in G$, we have
\[
h(F g) = \log N(\alpha^{(F g)}) = \log N\left(\bigvee_{k \in F g} k^{-1}
\alpha\right) = \log N\left(g^{-1} \alpha^{(F)}\right) =
\log N\left(\alpha^{(F)}\right)
\]
by applying~\eqref{e:card-g-alpha}.
This shows that $h$ is right $G$-invariant.  
 \end{proof}

Let $\alpha$ be a finite open cover of $X$.
Let $(F_j)_{j \in J}$ be a left F\o lner net for $G$.
By  Theorem~\ref{t:OW} and Lemma~\ref{l:h-foc-subadd}, the limit
\[
\htop(X,G,\alpha)\coloneqq \lim_{j \in J} \frac{\log N(\alpha^{(F_j)})}{|F_j|}
\]
exists,  is finite, and does not depend on the choice of the F\o lner net for $G$.
\par
We define the \emph{topological entropy} $ \htop(X,G) \leq \infty$ of the topological dynamical system $(X,G)$ by
\begin{equation}
\label{e:def-htop-X-G}
\htop(X,G) \coloneqq \sup_\alpha \htop(X,G,\alpha),  
\end{equation}
where $\alpha$ runs over all finite open covers of $X$.
Observe that
\[
\htop(X,G) = \lim_{\alpha \in \CC}  \htop(X,G,\alpha),
\]
where $\CC$ is the set consisting of all finite open covers of $X$ partially ordered by the relation defined by  
$\alpha \leq \beta$ if  $\beta$ is a refinement of $\alpha$.

\begin{theorem}
\label{t:properties-top-ent}
Let $G$ be an amenable group.
Then the following hold.
\begin{enumerate}[\rm (i)]
\item
Suppose that  $X$ is a topological space equipped with a continuous action of $G$ and
let  $Y \subset X$ be a closed invariant subset.
Then one has $\htop(Y,G) \leq \htop(X,G)$.
\item 
Let $X$ and $Y$ be topological spaces equipped with a continuous action of  $G$.
Suppose that the system $(Y,G)$ is a topological factor of $(X,G)$.
Then one has $\htop(Y,G) \leq \htop(X,G)$.
\end{enumerate} 
\end{theorem}

\begin{proof}
Let $\FF = (F_j)_{j \in J}$ be a left F\o lner net for $G$.
\par
(i) Consider  a finite open cover $\alpha$ of $Y$. 
For each $A \in \alpha$, we can find an open subset $A'$ of $X$ such that
$A = A' \cap Y$. Then $\alpha' \coloneqq \{A': A \in \alpha\} \cup \{X \setminus Y\}$ is a finite open cover of $X$
and
$N_Y(\alpha) \leq N_X(\alpha')$. Also, since $Y$ is $G$-invariant, so is $X \setminus Y$,
and, for every finite subset $F \subset G$, 
we have $\alpha^{(F)}= \left(\alpha'\right)^{(F)} \cap Y$.
\par
It follows that
\[
\htop(Y,G,\alpha) = \lim_{j \in J} \frac{\log N_Y(\alpha^{(F_j)})}{|F_j|}
\leq \lim_j \frac{\log N_X((\alpha')^{(F_j)})}{|F_j|}
= \htop(X,G,\alpha')
\leq \htop(X,G).
\]
Taking the supremum over all finite open covers $\alpha$ of $Y$, this gives us
$\htop(Y,G) \leq \htop(X,G)$.
\par
(ii)
As $(Y,G)$ is a topological factor of $(X,G)$,
there is a continuous $G$-equivariant surjective map $f \colon X \to Y$.
Let $\alpha$ be a finite open cover of $Y$.
Then $\alpha_X \coloneqq f^{-1}(\alpha) = \{f^{-1}(A): A \in \alpha\}$ is a finite open cover of $X$.
Moreover, if $\beta$ is a subcover of $\alpha$, then $\beta_X$ is a subcover of $\alpha_X$ and, since $f$ is surjective,
$N_Y(\alpha) = N_X(\alpha_X)$. Also, for every finite subset $F \subset G$ one has
$\left(\alpha^{(F)}\right)_X = \left(\alpha_X\right)^{(F)}$.
\par
We then have
\[
\htop(Y,G,\alpha) = \lim_{j \in J}  \frac{\log N_Y(\alpha^{(F_j)})}{|F_j|}
= \lim_{j \in J} \frac{\log N_X((\alpha_X)^{(F_j)})}{|F_j|} = \htop(X,G,\alpha_X)
\leq \htop(X,G).
\]
Taking the supremum over all finite open covers $\alpha$ of $Y$, 
we get $\htop(Y,G) \leq \htop(X,G)$.
\end{proof}

\subsection{$(F,U)$-separated subsets, $(F,U)$-spanning subsets, and $(F,U)$-covers}
Let $X$ be a compact uniform space equipped with a uniformly continuous action of a group $G$.
\par
Denote by  $\UU$ the directed set consisting of all entourages of $X$ partially ordered by reverse inclusion.
\par
Let $U \in \UU$  and let $F$ be a finite subset of $G$.
We define the entourage 
$U^{(F)}$ by 
\[
U^{(F)} \coloneqq \bigcap_{g \in F} g^{-1}  U.
\]
A subset $Z \subset X$ is said to be \emph{$(F,U)$-separated} if $(x,y) \notin U^{(F)}$
for all distinct $x,y \in Z$.
This amounts to saying that if $x,y \in Z$ satisfy $(g x, g y) \in U$ for all $g \in F$, then $x = y$.

\begin{lemma}
\label{l:sep-finite}
Every $(F,U)$-separated subset $Z \subset X$ is finite.
More precisely, there is an integer $N = N(X,F,U)$ such that every $(F,U)$-separated subset 
$Z \subset X$ has cardinality at most $N$.
\end{lemma}

\begin{proof}
Let $V$ be a symmetric entourage of $X$ such that  $V \circ V \subset  U^{(F)}$.
For each $x \in X$, the set $V[x]$ is a neighborhood of $x$.
By compactness of $X$, there is a finite subset $K \subset X$ such that $X = \bigcup_{x \in K} V[x]$.
Now let $Z \subset X$ be an $(F,U)$-separated subset.
As each $V[x]$ can contain at most one point of $Z$,
the set  $Z$ is finite with cardinality at most $N \coloneqq |K|$. 
\end{proof}

We define the integer $\sep(X,G,F,U)$ as being the maximal cardinality of an $(F,U)$-separated subset contained   in $X$:
\begin{equation}
\label{e:def-sep}
\sep(X,G,F,U) \coloneqq \max \{|Z| : \mbox{$Z \subset X$ is $(F,U)$-separated}\}.
\end{equation}

A subset $Z \subset X$ is said to be \emph{$(F,U)$-spanning} if for every $x \in X$ there exists $z = z(x) \in Z$ such that 
$(z,x) \in U^{(F)}$.
This amounts to saying that for every $x \in X$, there exists $z \in Z$ such that $(g z, g x) \in U$ for all $g \in F$.

\begin{lemma}
\label{l:finite-spanning}
There exists a finite $(F,U)$-spanning subset $Z \subset X$.
\end{lemma}

\begin{proof}
The set $U^{(F)}[x]$ is a neighborhood of $x$ in $X$ for each $x \in X$.
By compactness of $X$, 
 there exists a finite subset $Z \subset X$ such that $X = \bigcup_{z \in Z} U^{(F)}[z]$.
Then $Z$ is a finite $(F,U)$-spanning subset for $X$.
\end{proof}

We define the integer $\spa(X,G,F,U)$ as being the minimal cardinality of an $(F,U)$-spanning subset for  $X$: 
\begin{equation}
\label{e:def-spa}
\spa(X,G,F,U) \coloneqq \min \{|Z | : \mbox{$Z \subset X$ is $(F,U)$-spanning}\}.
\end{equation}

A  cover $\alpha$ of $X$ is called an \emph{$(F,U)$-cover}
if for each $A \in \alpha$ and all $x,y \in A$ one has $(x,y) \in U^{(F)}$, that is,
$(g x, g y) \in U$ for all $g \in F$.

\begin{lemma}
\label{l:finite-covering}
There exists a finite $(F,U)$-cover $\alpha$ of $X$.
\end{lemma}

\begin{proof}
Let $V$ be a symmetric entourage of $X$ such that $V \circ V \subset U$.
The set $A_x \coloneqq V^{(F)}[x]$ is a neighborhood of $x$ in $X$ for each $x \in X$.
By compactness, we can find a finite subset $Z \subset X$ such that 
the set $\alpha \coloneqq
\{A_z :z \in Z \}$ is a cover of $X$. If $z \in Z$ and $x,y \in A_z$ we have $(x,z), (z,y) \in V^{(F)}$
so that $(x,y) \in V^{(F)} \circ V^{(F)} \subset (V \circ V)^{(F)} \subset U^{(F)}$.
Therefore  $\alpha$ is a finite $(F,U)$-cover of $X$.
\end{proof}

We define the integer $\cov(X,G,F,U)$ as being the minimal cardinality of an $(F,U)$-cover of $X$.

\begin{lemma}
\label{l:sep-spa-cov-right-inv}
Let $g \in G$.
Then one has
\begin{equation}
\label{e:sep-spa-cov-right-inv}
\begin{split}
\sep(X,G,Fg,U) & = \sep(X,G,F,U)\\
\spa(X,G,F g,U) & = \spa(X,G,F,U)\\
\cov(X,G,F g,U) & = \cov(X,G,F,U).
\end{split}
\end{equation}
\end{lemma}

\begin{proof}
A subset  $Z \subset X$ is $(F,U)$-separated if and only if  $g^{-1} Z$ is $(F g, U)$-separated.
As $|Z| = |g^{-1} Z|$, this gives us the first equality.
Similarly, the second one follows from the fact that  $Z \subset X$ is $(F,U)$-spanning if and only if $g^{-1} Y$ is $(F g, U)$-spanning.
The last equality follows from the fact that a cover $\alpha$ of $X$ is an $(F,U)$-cover if and only if  $g^{-1}\alpha$ is an $(Fg,U)$-cover.
\end{proof}

\begin{lemma}
\label{l:sep-spa-cov-decrease}
The maps $U \mapsto \sep(X,G,F,U)$,
$U \mapsto \spa(X,G,F,U)$, and
$U \mapsto \cov(X,G,F,U)$ are non-decreasing on $\UU$.
\end{lemma}

\begin{proof}
It suffices to observe that if $U_1,U_2 \in \UU$ satisfy $U_1 \subset U_2$, then
every $(F,U_2)$-separated subset of $X$ is also $(F,U_1)$-separated,
 every $(F,U_1)$-spanning subset of $X$ is also $(F,U_2)$-spanning,
and every $(F,U_1)$-cover of $X$ is also an $(F,U_2)$-cover. 
\end{proof}

\begin{lemma}
\label{l:inegalites-entre-cov-sep-spa}
Let $U$ and $V$ be entourages of $X$ such that $V$ is symmetric and $U \circ U^* \subset V$.
Then one has
\begin{multline}
\label{e:relations-spa-sep-cov}
\cov(X,G,F,V \circ V) \leq \spa(X,G,F,V) \leq \sep(X,G,F,V) \\
\leq \spa(X,G,F,U) \leq \cov(X,G,F,U).
\end{multline}
\end{lemma}

\begin{proof}
Let $Y \subset X$ be an $(F,V)$-spanning subset with minimal cardinality.
Then $\alpha \coloneqq \{V^{(F)}[y] : y \in Y \}$ is a finite cover of $X$.
Moreover, if $x,x' \in V^{(F)}[y]$, then  $(y,x), (y,x') \in V^{(F)}$
and hence $(x,x') \in   V^{(F)} \circ V^{(F)} \subset (V \circ V)^{(F)} $.
This shows that $\alpha$ is an $(F,V \circ V)$-cover of $X$.
We deduce that $\cov(X,G,F,V \circ V) \leq |\alpha| \leq |Y| = \spa(X,G,F,V)$.
\par
Let now $Z \subset X$ be an $(F,V)$-separated subset with maximal cardinality.
If $y \in X \setminus Z$, then $y \in V^{(F)}[z]$ for some $z \in Z$ by maximality.
Thus $Z$ is an $(F,V)$-spanning subset of $X$ and hence
$\spa(X,G,F,V) \leq |Z| = \sep(X,G,F,V)$.
\par
Suppose now that $S \subset X$ is an $(F,U)$-spanning subset for $X$ with minimal cardinality.
Thus for each $x \in X$ there exists $s(x) \in S$ such that $(gs(x),gx) \in U$ for all $g \in F$.
If $z_1, z_2 \in Z$ are distinct, then $s(z_1) \neq s(z_2)$, otherwise, since $U \circ U^* \subset V$,
we would have $(gz_1, gz_2) \in V$ for all $g \in F$, contradicting the fact that $Z$ is $(F,V)$-separated.
It follows that $\sep(X,G,F,V) = |Z| \leq |S| = \spa(X,G,F,U)$.
\par
Finally, let $\alpha$ be an $(F,U)$-cover of $X$ with minimal cardinality. For each $A \in \alpha$ pick
$t_A \in A$ and set $T \coloneqq \{t_A: A \in \alpha\}$. Given $x \in X$ we can find $A \in \alpha$
such that $x \in A$. Since $\alpha$ is an $(F,U)$-cover, we have $(t_A,x) \in U^{(F)}$.
This shows that $T$ is an $(F,U)$-spanning subset. We deduce that
$\spa(X,G,F,U) \leq |T| \leq |\alpha| = \cov(X,G,F,U)$.
\end{proof}

\begin{lemma}
\label{l:cov-sub-multipli}
Let $U \in \UU$  and let $E,F$ be two finite subsets of $G$. Then
\begin{equation}
\label{e:cov-union}
\cov(X,G,E \cup F,U) \leq \cov(X,G,E,U) \cdot \cov(X,G,F,U).
\end{equation}
\end{lemma}
\begin{proof}
Let $\alpha$ (resp.\ $\beta$) be an $(E,U)$-cover (resp.\ $(F,U)$-cover) of $X$ with minimal cardinality.
Let $A \in \alpha$ and  $B \in \beta$. 
If $x, y \in A \cap B$,
then 
$(x,y) \in U^{(E)} \cap U^{(F)} = U^{(E \cup F)}$.
Therefore $\alpha \vee \beta$ is an $(E \cup F, U)$-cover of $X$. 
It follows that $\cov(X,G,E \cup F,U)  \leq |\alpha \vee \beta| \leq |\alpha| \cdot |\beta| = \cov(X,G,E,U) \cdot \cov(X,G,F,U)$.
\end{proof}

\subsection{Uniform approaches to topological entropy}
The computation of topological entropy  for uniformly continuous actions of amenable groups on compact uniform spaces we discuss now is a uniform analogue
of the Bowen-Dinaburg approach in the metrizable case~\cite{bowen-entropy-group-endo-1971},
\cite{dinaburg}.   
\par
Let $X$ be a compact uniform space equipped with a uniformly continuous action of an amenable  group $G$.
Let $\FF  = (F_j)_{j \in J}$ be a left F\o lner net for $G$.
\par
Let $U$ be an entourage of $X$.
We define the quantities $\hsep(X,G,\FF,U)$ and $\hspa(X,G,\FF,U) $ by
\begin{equation}
\label{d:hsep-U}
\hsep(X,G,\FF,U) \coloneqq \limsup_{j \in J} \frac{\log \sep(X,G,F_j,U)}{|F_j|}
\end{equation}
and 
\begin{equation}
\label{d:hspa-U}
\hspa(X,G,\FF,U) \coloneqq  \limsup_{j \in J} \frac{\log \spa(X,G,F_j,U)}{|F_j|}.
\end{equation}
It follows from Lemma~\ref{l:sep-spa-cov-right-inv} and Lemma~\ref{l:cov-sub-multipli} that the map
$F \mapsto \log \cov(X,G,F,U)$ is right-invariant and subadditive on $\PPf(G)$.
Thus, we deduce from Theorem~\ref{t:OW}
that the limit
\begin{equation}
\label{e:def-hcov}
\hcov(X,G,U) \coloneqq \lim_{j \in J} \frac{\log \cov(X,G,F_j,U)}{|F_j|} 
\end{equation}
exists, is finite, and does not depend on the choice of the left F\o lner net $\FF$.

\begin{lemma}
\label{l:inegalites-entre-cov-sep-spa-h}
Let $U$ and $V$  be  entourages of $X$ such that $V$ is symmetric and $U \circ U^* \subset V$.
Then one has
\begin{multline}
\label{e:relations-spa-sep-cov-h}
\hcov(X,G,V \circ V)
\leq \hspa(X,G,\FF,V)
\leq \hsep(X,G,\FF,V)\\
\leq \spa(X,G,\FF,U)
\leq \hcov(X,G,U).
\end{multline}
\end{lemma}

\begin{proof}
This is immediately deduced from Lemma~\ref{l:inegalites-entre-cov-sep-spa}
after taking limits.
\end{proof}

\begin{lemma}
\label{l:h-sep-spa-cov-decrease}
The maps $U \mapsto \hsep(X,G,\FF,U)$,
$U \mapsto \hspa(X,G,\FF,U)$, and
$U \mapsto \hcov(X,G,U)$ are non-decreasing on $\UU$.
\end{lemma}

\begin{proof}
This immediately follows from Lemma~\ref{l:sep-spa-cov-decrease} by taking limits.
\end{proof}

 By Lemma~\ref{l:h-sep-spa-cov-decrease},
 the maps 
$U \mapsto \hsep(X,G,\FF,U)$,
$U \mapsto \hspa(X,G,\FF,U)$,
and $U \mapsto \hcov(X,G,U)$  
admit (finite or infinite) limits on $\UU$ that are equal to their upperbounds, namely
\begin{align}
\label{e:def-hsep-X-G} \hsep(X,G,\FF) &\coloneqq \lim_{U \in \UU} \hsep(X,G,\FF,U) = \sup_{U \in \UU} \hsep(X,G,\FF,U), \\
\label{d:hspa}
\hspa(X,G,\FF) &\coloneqq \lim_{U \in \UU} \hspa(X,G,\FF,U) = \sup_{U \in \UU} \hspa(X,G,\FF,U), \\
\hcov(X,G) &\coloneqq \lim_{U \in \UU}  \hcov(X,G,U) = \sup_{U \in \UU}  \hcov(X,G,U).
\end{align} 

\begin{lemma}
\label{l:hc-hs-equal}
One has
\begin{equation}
\label{e:hc-hs-equal}
\hsep(X,G,\FF) = \hspa(X,G,\FF) = \hcov(X,G).
\end{equation}
\end{lemma}

\begin{proof}
Equalities~\eqref{e:hc-hs-equal} follow from~\eqref{e:relations-spa-sep-cov-h}
after taking limits in $\UU$.
\end{proof}

We deduce from Lemma~\ref{l:hc-hs-equal} that
$\hsep(X,G,\FF)$ and $\hspa(X,G,\FF)$ do not depend on the choice of the left F\o lner net $\FF$.
In the sequel, we shall simply write
$\hsep(X,G)$ instead of $\hsep(X,G,\FF)$ and $\hspa(X,G)$ instead of $\hspa(X,G)$.

\begin{theorem}
\label{t:top-ent-uni}
Let $X$ be a compact uniform space equipped with a uniformly continuous action of an amenable group $G$. Then one has
\begin{equation}
\label{e:htop-equals-huni}
\htop(X,G) = 
\hsep(X,G)= 
\hspa(X,G)= \hcov(X,G).
\end{equation}
\end{theorem}

For the proof, we shall need some auxiliary results.
The first one is a uniform version of Lebesgue's covering lemma 
(see e.g.~\cite[Theorem 0.20]{walters-ergodic}).

\begin{lemma}
\label{l:Lebesgue-covering}
Let $X$ be a compact uniform space. Let $\alpha$ be an open cover of $X$. Then there exists an entourage $U = U(\alpha)$ of $X$
such that for every $x \in X$ there exists $A \in \alpha$ such that $U[x] \subset A$.
\end{lemma}

Such an entourage $U$  is then called a \emph{Lebesgue entourage} for the open cover $\alpha$.

\begin{proof}
Since $\alpha$ is an open cover of $X$, there exists, for every point $x \in X$, 
an open subset $A_x \in \alpha$ and an entourage $V_x$ of $X$ such that
$V_x[x] \subset A_x$.
Choose, for each $x \in X$, an entourage $W_x$ such that $W_x \circ W_x \subset V_x$.
By compactness of $X$, there exists a finite subset $Y \subset X$ such that the sets $W_y[y]$, $y \in Y$,  cover $X$.
Consider the entourage $U  \coloneqq \bigcap_{y \in Y} W_y$.
Then, for every $x \in X$, there is a point $y \in Y$ such that $x \in W_y[y]$.
 This implies
 \[ 
 U[x] \subset (U \circ W_y)[y] \subset (W_y \circ W_y)[y] \subset V_y[y] \subset A_y.
 \]
 Thus the entourage $U$ has the required property. 
\end{proof}

\begin{lemma}
\label{l:N-sep-spa}
Let $X$ be a compact uniform  space equipped with a uniformly  continuous action of a group $G$. 
Let $\alpha$ be a finite  open cover of $X$
and let $F$ be a finite subset of $G$.
Suppose that  $U$ is a Lebesgue entourage for $\alpha$.
 Then one has
\begin{equation*}
\label{e:N-alpha-F-spa}
N(\alpha^{(F)}) \leq \spa(X,G,F,U).
\end{equation*}
\end{lemma}

\begin{proof}
Let $Z \subset X$ be an $(F,U)$-spanning subset with minimal cardinality. 
Since $U$ is a Lebesgue entourage for $\alpha$, given $z \in Z$ and $g \in F$,  we can find   an open set $A_{z,g} \in \alpha$ 
such that $U[gz] \subset A_{z,g}$. 
As $Z$ is $(F,U)$-spanning, given any point $x \in X$ we can find $z = z(x) \in Z$ 
such that $(gz,gx) \in U$ for all $g \in F$. 
This implies $gx \in U[gz] \subset A_{z,g}$ for all $g \in F$, that is,
$x \in \bigcap_{g \in F} g^{-1}A_{z,g}$. 
This shows that $\{\bigcap_{g \in F}  g^{-1} A_{z,g}: z \in Z\}$ is a subcover of $\alpha^{(F)}$.
Therefore $N(\alpha^{(F)}) \leq |Z| = \spa(X,G,F,U)$.
\end{proof}

\begin{lemma}
\label{l:sep-less-N-alpha-F}
Let $X$ be a compact uniform  space equipped with a uniformly  continuous action of a group $G$.
Let $U$ and $V$ be entourages of $X$ with $V$ symmetric and $V \circ V \subset U$. 
Let $\alpha$ be a finite  open cover of $X$
and let $F$ be a finite subset of $G$.
Suppose that for every $A \in \alpha$, there exists $x \in X$ such that $A \subset V[x]$. 
 Then one has
 \begin{equation*}
 \sep(X,G,F,U) \leq N(\alpha^{(F)}).
 \end{equation*}
\end{lemma}

\begin{proof}
Let  $Z \subset X$ be an $(F,U)$-separated subset with maximal cardinality.
It follows from the hypotheses that each open set in $\alpha^{(F)}$ can contain at most one point of $Z$.
Therefore
$\sep(X,G,F,U) = |Z| \leq N(\alpha^{(F)})$.   
\end{proof}

\begin{proof}[Proof of Theorem~\ref{t:top-ent-uni}]
Let $\FF = (F_j)_{j \in J}$ be a left F\o lner net for $G$.
\par
Let $\alpha$ be a finite open cover of $X$.
Choose a Lebesgue entourage $U$ for $G$.
By Lemma~\ref{l:N-sep-spa}, we have
\[
N(\alpha^{(F_j)}) \leq \spa(X,G,F_j,U)
\]
for all $j \in J$.
It follows that
\begin{align*}
\htop(X,G,\alpha) &= \lim_{j \in J} \frac{\log N(\alpha^{(F_j)})}{|F_j|} \\
&\leq \limsup_{j \in J} \frac{\log \spa(X,G,F_j,U)}{|F_j|} \\ 
&= \hspa(X,G,\FF,U) \\ 
&\leq \hspa(X,G) 
\end{align*}
and hence
\begin{equation}
\label{e:htop-less-hspa}
\htop(X,G) = \sup_{\alpha} \htop(X,G,\alpha) \leq \hspa(X,G).
\end{equation}
\par
Let now $U$ be an entourage of $X$.
Choose a symmetric entourage $V$ of $X$ such that $V \circ V \subset U$.
For each $x \in X$, there is an open neighborhood $A_x$ of $x$ such that $A_x \subset V[x]$.
By compactness of $X$, we can find a finite subset $K \subset X$ such that $\alpha \coloneqq \{A_x : x \in K\}$ covers $X$.
Lemma~\ref{l:sep-less-N-alpha-F} gives us
\[
\sep(X,G,F_j,U) \leq N(\alpha^{(F_j)})
\]
for all $j \in J$.
Therefore
\begin{multline*}
\hsep(X,G,\FF,U) = \limsup_{j \in J} \frac{\log \sep(X,G,F_j,U)}{|F_j|}\\
\leq \lim_{j \in J} \frac{\log N(\alpha^{(F_j)})}{|F_j|} =
\htop(X,G,\alpha) \leq \htop(X,G),
\end{multline*}
and hence
\begin{equation}
\label{l:hsep-less-htop}
\hsep(X,G) = \sup_{U \in \UU} \hsep(X,G,\FF,U) \leq \htop(X,G).
\end{equation}
Combining Lemma~\ref{l:hc-hs-equal} with inequalities~\eqref{e:htop-less-hspa} and~\eqref{l:hsep-less-htop},
we finally get~\eqref{e:htop-equals-huni}. 
\end{proof}

\subsection{Topological entropy and expansiveness}

\begin{theorem}
\label{t:entropy-expansive}
Let $X$ be a compact Hausdorff  space equipped with an expansive continuous  action of an amenable group $G$. 
Let $U_0$ be a closed expansiveness entourage for $(X,G)$.
Then one has
\begin{equation}
\label{e:entropie-expansive}
\htop(X,G)  =  \hcov(X,G,U_0).
\end{equation} 
\end{theorem}

In the proof of Theorem~\ref{t:entropy-expansive}, we shall use the following result.

\begin{lemma}
\label{l:finite-times-folner}
Let $G$ be an amenable group and let $(F_j)_{j \in J}$ be a left F\o lner net for $G$.
Suppose that $E \subset G$ is a nonempty finite subset.
Then the following hold:
\begin{enumerate}[\rm (i)]
\item
$\lim_{j \in J} \frac{|E F_j \setminus F_j|}{|F_j|} = 0$.
\item
$\lim_{j \in J} \frac{|E F_j |}{|F_j|} = 1$;
\item
the net $(E F_j)_{j \in J}$ is a left F\o lner net for $G$.
\end{enumerate}
\end{lemma}

\begin{proof}
As $E F_j \setminus F_j = \bigcup_{g \in E} (g F_j \setminus F_j)$, we have
\[
|E F_j \setminus F_j| = \left| \bigcup_{g \in E} (g F_j \setminus F_j) \right| 
\leq \sum_{g \in E} | g F_j \setminus F_j|,
\]
so that Assertion~(i) follows from~\eqref{e:def-left-Folner-net}.
\par
As the set $E$ is nonempty, we have $|E F_j| \geq |F_j|$ and hence $|E F_j \setminus F_j| \geq |F_j \setminus E F_j|$.
Therefore, we deduce from (i) that
\[
\lim_{j \in J} \frac{| F_j \setminus E F_j|}{|F_j|} = 0.
\]
Now, using $|E F_j| = |F_j| + |E F_j \setminus F_j|  - |F_j \setminus E F_j|$,
we get
\[
\lim_{j \in J} \frac{|E F_j |}{|F_j|} = 1 + \lim_{j \in J} \frac{|E F_j  \setminus F_j|}{|F_j|}  
- \lim_{j \in J} \frac{| F_j \setminus E F_j|}{|F_j|} = 1,
\]
which gives  (ii).
\par   
Let $g \in G$ and
fix some arbitrary element $e \in E$. 
By applying (i)
with $E$ replaced by $e^{-1} g E$, we obtain
\begin{equation}
\label{e:lim-einv-gE-Fj}
\lim_{j \in J} \frac{|e^{-1} g E F_j \setminus F_j|}{|F_j|} = 0. 
\end{equation}
Now observe that
\[
|g E F_j \setminus E F_j| = |e^{-1} (g E F_j \setminus E F_j)| = |e^{-1} g E F_j \setminus e^{-1} E F_j | \leq |e^{-1} g E F_j \setminus F_j|,
\]
where the last inequality follows from the fact that
$F_j \subset  e^{-1} E F_j$.
Thus, we deduce from~\eqref{e:lim-einv-gE-Fj} that
\begin{equation*}
\lim_{j \in J} \frac{| g E F_j \setminus E F_j|}{|F_j|} = 0. 
\end{equation*}
Using (ii), this gives us
\[
\lim_{j \in J} \frac{|g E F_j \setminus E F_j|}{|E  F_j|}
= \lim_{j \in J} \left( \frac{|E  F_j|}{|  F_j|} \right)^{-1} \cdot  \frac{|g E F_j \setminus E F_j|}{|F_j|}
= 0,  
\]  
which shows (iii).
\end{proof} 

\begin{proof}[Proof of Theorem~\ref{t:entropy-expansive}] 
Let $U$ be an entourage of $X$. 
By Lemma~\ref{l:expansive-base-entourages},
there exists a nonempty finite subset $E \subset G$
such that $U_0^{(E)} \coloneqq \bigcap_{g \in E} g^{-1} U_0 \subset U$.
\par
Observe that if $F \subset G$ is a finite subset,
then a cover $\alpha$ of $X$ is an $(F,U_0^{(E)})$-cover
if and only if $\alpha$ is an $(EF,U_0)$-cover.
Therefore
\begin{equation}
\label{e:cov-F-EF}
\cov(X,G,F,U_0^{(E)}) = \cov(X,G,E F,U_0).
\end{equation}
 Let now $ (F_j)_{j \in J}$ be a left-Folner net for $G$. 
Then we have
\begin{align*}
\hcov(X,G,U) &\leq \hcov(X,G,U_0^{(E)})
&&\text{(by Lemma~\ref{l:h-sep-spa-cov-decrease})} \\
&= \lim_{j \in J} \frac{\log \cov(X,G,F_j,U_0^{(E)})}{|F_j|} 
&&\text{(by~\eqref{e:def-hcov})} \\
&= \lim_{j \in J} \frac{\log \cov(X,G,E F_j,U_0)}{|F_j|} 
&&\text{(by~\eqref{e:cov-F-EF})} \\
&= \lim_{j \in J} \frac{\log \cov(X,G,E F_j,U_0)}{|E F_j|} 
&&\text{(by Lemma~\ref{l:finite-times-folner}.(ii))} \\
&= \hcov(X,G,U_0),
\end{align*}
where the last equality follows from the fact that the net $(E F_j)_{j \in J}$ is a left F\o lner net by Lemma~\ref{l:finite-times-folner}.(iii).
As
\[
\htop(X,G) = \hcov(X,G) = \sup_U \hcov(X,G,U),
\]
by Theorem~\ref{t:top-ent-uni},
this shows that $\htop(X,G) = \hcov(X,G,U_0)$.
\end{proof}

Since $\hcov(X,G,U_0) < \infty$ by the Ornstein-Weiss lemma,
an immediate consequence of  Theorem~\ref{t:entropy-expansive} is the following.

\begin{corollary}
\label{c:ent-expansive-positive}
Let $X$ be a compact Hausdorff  space equipped with an expansive continuous  action of an amenable
group $G$.
Then  $\htop(X,G) < \infty$.
\qed 
\end{corollary}

\begin{example}
Let $G$ be an amenable group, $A$ a finite set, and $X \subset A^G$ a subshift.
Equip $X$ with the action of $G$ induced by the $G$-shift on $A^G$.
Then $X$ is compact Hausdorff and $(X,G)$ is expansive.
A closed expansiveness entourage for $(X,G)$ is $U_0 \coloneqq \{(x,y) \in X \times X : x(1_G) = y(1_G)\}$.
Observe that if $F \subset G$ is a finite subset, then
$(x,y) \in U_0^{(F)}$ if and only if $x(g) = y(g)$ for  all $g \in F^{-1}$.
It follows that
\[
\cov(X,G,F,U_0) = |\pi_{F^{-1}}(X)|,
\]
where $\pi_{F^{-1}} \colon A^G \to A^{F^{-1}}$ is the projection map.
We deduce from Theorem~\ref{t:entropy-expansive} that
if $(F_j)_{j \in J}$ is a left F\o lner net for $G$, then 
\[
\htop(X,G) = \lim_{j \in J} \frac{\vert\pi_{F_j^{-1}}(X)\vert}{|F_j|}
\]
(cf.~\cite{book}, \cite[Proposition~2.7]{csc-ijm-expansive}). This gives us $\htop(X,G) = \log |A|$ if $X = A^G$ is the full shift.
\end{example}

\begin{corollary}
\label{c:sep-spa-expansiveness}
Let $X$ be a compact Hausdorff space equipped with an expansive continuous action of an amenable group $G$.
Let $\FF$ be a left F\o lner net for $G$ and let $V_0$ be a closed symmetric entourage of $X$ such that
$V_0 \circ V_0$ is an expansiveness entourage for $(X,G)$.
Then one has
\begin{equation}
\label{e:sep-spa-expansiveness}
\htop(X,G) = \hsep(X,G,\FF,V_0) = \hspa(X,G,\FF,V_0).
\end{equation}
In particular, the quantities $\hsep(X,G,\FF,V_0)$ and $\hspa(X,G,\FF,V_0)$ do not depend on the left F\o lner net $\FF$.
\end{corollary}

\begin{proof}
Let $U_0$ be a closed expansiveness entourage such that $U_0 \circ U_0^* \subset V_0$.
By applying Lemma~\ref{l:inegalites-entre-cov-sep-spa-h} with $V = V_0$ and $U = U_0$,
we get
\begin{multline*}
\htop(X,G) = \hcov(X,G, V_0 \circ V_0) \leq \hspa(X,G,\FF,V_0) \leq \hsep(X,G,\FF,V_0)\\
\leq \hspa(X,G,\FF,U_0) \leq \hcov(X,G,U_0) = \htop(X,G),
\end{multline*}
where the first and  last equalities follow from Theorem~\ref{t:entropy-expansive}.
\end{proof}

\section{Homoclinicity}
\label{sec:homoclinicity}

Let $X$ be a uniform space equipped with a uniformly continuous  action of a group $G$.
One says that the points $x,y \in X$ are \emph{homoclinic} 
if they satisfy the following condition:
for every entourage $U$ of $X$, there exists a finite subset $\Omega = \Omega(U) \subset G$ such that
$(gx,gy) \in U$ for all $g \in G \setminus \Omega$.

\begin{proposition}
Let $X$ be a uniform space equipped with a uniformly continuous   action of a group $G$.
Then homoclinicity is an equivalence relation on $X$. 
\end{proposition}

\begin{proof}
Homoclinicity is a reflexive relation since every entourage of $X$ contains the diagonal by (UNI-1).
It is symmetric since every entourage contains  a symmetric one by (UNI-3) and (UNI-4).
Transitivity follows from the fact that if $U$ is an entourage of $X$, then
 there is an entourage $V$  of $X$  such that
$V\circ V \subset U$ by (UNI-5).  
\end{proof}

\begin{proposition}
\label{p:homoclinicity-almost-equality}
Let $G$ be a group and let $A$ be a set.
Equip $A^G$ with its prodiscrete uniform structure and the shift action of $G$.
Then two configurations in $A^G$ are homoclinic 
if and only if they are almost equal. 
\end{proposition}

\begin{proof}
Let $x,y \in A^G$ and suppose first that $x$ and $y$ are almost equal.
This means that there exists a finite subset $\Omega \subset G$ such that
$x(g) = y(g)$ for all $g \in G \setminus \Omega$.
Now let $W \subset A^G \times A^G$ be an entourage of $A^G$.
By definition of the prodiscrete uniform structure,
there is a finite subset $\Lambda \subset G$ such that $W(\Lambda) \subset W$,
where $W(\Lambda) \subset A^G \times A^G$ is  as in \eqref{e:base-entourages}.
Observe now that $(gx,gy) \in W(\Lambda) \subset W$ for all $g \in G \setminus \Lambda \Omega^{-1}$.
As the set $\Lambda \Omega^{-1}$ is finite, 
we deduce that $x$ and $y$ are homoclinic. 
 \par
Conversely, suppose that  $x$ and $y$ are homoclinic.
Then there exists a finite subset $\Omega \subset G$ such that $(g x, g y) \in W(\{1_G\})$ for all 
$g \in G \setminus \Omega$.
This implies that $x(g) = y(g)$ for all $g \in G \setminus \Omega^{-1}$.
Therefore  $x$ and $y$ are almost equal.
\end{proof}

The following result is a uniform version of \cite[Lemma 6.2]{chung-li}
(see also \cite[Lemma~1]{bryant} and \cite[Theorem~10.36]{gott-hed}).

\begin{proposition}
\label{p:homoclinic-equiv}
Let $X$ be a compact Hausdorff space equipped with a continuous action of a group $G$.
Suppose that the action of $G$ on $X$ is expansive and let $U_0$ be a closed expansiveness entourage for 
$(X,G)$.
Let $x,y \in X$. Then the following conditions are equivalent:
\begin{enumerate}[{\rm (a)}]
\item $x$ and $y$ are homoclinic;
\item there exists a finite subset $\Omega_0 \subset G$ such that
$(gx, gy) \in U_0$ for all $g \in G \setminus \Omega_0$.
\end{enumerate}
\end{proposition}

\begin{proof}
The implication (a) $\implies$ (b) follows immediately from the definition of homoclinicity.
\par
Conversely, suppose (b) and  let $U$ be an entourage of $X$. 
By Lemma~\ref{l:expansive-base-entourages}, there exists a finite subset $F = F(U) \subset G$ such that 
$\bigcap_{h \in F} h^{-1} U_0 \subset U$. 
Consider the finite subset $\Omega \subset G$ defined by $\Omega \coloneqq F^{-1}\Omega_0$.
Then for $g \in G \setminus \Omega$ and $h \in F$, we have 
$h g \in G \setminus \Omega_0$ so that $(g x, g y) \in h^{-1}  U_0$.
 It follows that $(g x, g y) \in \bigcap_{h \in F} h^{-1} U_0 \subset U$ 
 for all $g \in G \setminus \Omega$.
This shows that $x$ and $y$ are homoclinic.
\end{proof}

\section{Weak specification}
\label{sec:weak-specification} 

The concept of specification  for dynamical systems was first introduced by Bowen 
in~\cite[Section~2.9]{bowen-tams-1971}.
Several variants and extensions  of Bowen's original definition of specification  appear in the literature
(see  \cite[Definition~6.1]{chung-li}, \cite[Chapter~21]{dgs}, \cite{klo-specification}, 
\cite[Definition~5.1]{lind-schmidt-jams-1999}, \cite{ruelle-tams-1973}).
Roughly speaking, specification is a property allowing  to  approximate sufficiently separated pieces of orbits  by a single (sometimes required to be periodic) orbit.  
Connections of specification with chaos for iterates of uniformly continuous maps on uniform spaces was investigated in~\cite{das-das}.
The definition of weak specification below  is 
equivalent to the one given in~\cite{chung-li} and \cite{li-goe} when restricted  to continuous group actions on compact metrizable spaces.   

\begin{definition}
\label{d:weak-specification}
Let $X$ be a uniform space equipped with a uniformly continuous action of a group $G$. 
We say that the action of $G$ on $X$ has the \emph{weak specification property} if it satisfies the following condition:
\begin{enumerate}[{\rm (WSP)}]
\item for every entourage $U$ of $X$, there exists  a finite subset $\Lambda = \Lambda(U) \subset G$ such that the following holds:
for any finite family $(\Omega_i)_{i \in I}$ of finite subsets of $G$ such that
$\Omega_j \cap \Lambda \Omega_k = \varnothing$ for all distinct $j,k \in I$  and for any family of points $(x_i)_{i \in I}$ in $X$,
there exists a point $x \in X$ such that 
$( x, x_i) \in U^{(\Omega_i)}$ for all   $i \in I$. 
\end{enumerate}
Such a subset  $\Lambda \subset G$ is then called a \emph{specification subset} for $(X,G,U)$.
\end{definition}

\begin{proposition}
\label{p:WSP-factor}
Let $X$ be a uniform space equipped with a uniformly continuous action of a group $G$.
Suppose that the action of $G$ on $X$ has the weak specification property.
Then every uniform factor of $(X,G)$ also has the weak specification property.
\end{proposition}

\begin{proof}
Let $Y$ be a uniform space equipped with a uniformly continuous action of $G$ and suppose
that $(Y,G)$ is a uniform factor of $(X,G)$. 
This means that there
exists a $G$-equivariant uniformly continuous surjective map $f \colon X \to Y$. 
\par
Let $V$ be an entourage of $Y$. 
Since $f$ is uniformly continuous, we can find an entourage $U$ of $X$ such that
\begin{equation}
\label{e:f-f-U-V}
(f \times f) (U) \subset V.
\end{equation}
Let $\Lambda  \subset G$ 
be a specification subset for $(X,G,U)$. Let $(\Omega_i)_{i \in I}$ be a finite family of finite subsets of $G$ such that
$\Omega_j \cap \Lambda \Omega_k = \varnothing$ for all distinct $j,k \in I$, and let $(y_i)_{i \in I}$ be a family of points in $Y$.
Since $f$ is surjective, we can find, for each $i \in I$,  a point $x_i \in X$ such that $y_i = f(x_i)$.
On the other hand, as $\Lambda$ is a specification subset for $(X,G,U)$, 
there exists  a point $x \in X$  such that
\begin{equation}
\label{e:gixi-gix-U}
(x, x_i) \in U^{(\Omega_i)}
\end{equation}
for all  $i \in I$.
Setting $y \coloneqq f(x) \in Y$, we then have
\begin{align*}
(y, y_i) & = (f(x),  f(x_i))\\
&= (f \times f) (x,x_i) \\
&\in (f \times f) (U^{(\Omega_i)}) 
&& \text{(by~\eqref{e:gixi-gix-U})} \\
&= ((f \times f)(U))^{(\Omega_i)}
&&\text{(since $f$ is $G$-equivariant)} \\
&\subset V^{(\Omega_i)}
&&\text{(by \eqref{e:f-f-U-V})}
\end{align*}
and hence $(y,y_i) \in V^{(\Omega_i)}$ for all  $i \in I$. 
This shows that $\Lambda$ is a specification subset for $(Y,G,V)$ and hence that the action of $G$ on $Y$ has the weak specification property.
\end{proof}

\begin{proposition}
\label{p:WSP-produit}
Let $(X_k)_{k \in K}$ be a (possibly infinite) family of uniform spaces and let $G$ be a group.
Suppose that each $X_k$, $k \in K$, is equipped with an action of $G$ having the weak specification property.
Then the diagonal action of $G$ on the uniform product $X \coloneqq \prod_{k \in K} X_k$ 
also has the weak specification property.
\end{proposition}

\begin{proof}
For each $k \in K$, denote by $p_k \colon X \to X_k$  the projection map onto $X_k$.
Let $U$ be an entourage of $X$.
Then one can find  a finite set $L \subset K$ and, for every $k \in L$, and entourage $U_k$ of $X_k$ such that the entourage $V$ of $X$ defined by
\[
V \coloneqq \{ (x,y) \in X \times X : (p_k(x),p_k(y))  \in U_k \text{  for all  } k \in L\} 
\]
satisfies $V \subset U$.
As the action of $G$ on $X_k$ has the weak specification property, one can find, for every $k \in L$,
a finite subset $\Lambda_k \subset G$
such that $\Lambda_k$ is a specification subset for $(X_k,G,U_k)$.
Clearly $\Lambda \coloneqq \bigcup_{k \in L} \Lambda_k$ is a specification subset for 
$(X,G,V)$ and hence for $(X,G,U)$.
\end{proof}

The following result will be used in the proof of Theorem~\ref{t:entropie-diminue-failure-preinjectivite}.  It says  that, when $X$ is compact Hausdorff,  
we get an equivalent definition for weak specification  if we remove the 
finiteness hypotheses for the index set $I$ and the  subsets 
$\Omega_i$  in Definition~\ref{d:weak-specification}.

\begin{proposition}
\label{p:wspec-infinite}
Let $X$ be a compact Hausdorff space equipped with a  continuous action of a group $G$ satisfying the weak specification property.
Let $U$ be a closed entourage of $X$ and let $\Lambda \subset G$ be a specification subset for $(X,G,U)$.
Let $(\Omega_i)_{i \in I}$ be a (possibly infinite) family  of (possibly infinite) subsets of $G$ such that
\begin{equation}
\label{e:lambda-separated}
\Omega_j \cap \Lambda \Omega_k = \varnothing \text{  for all distinct $j,k \in I$.}
\end{equation} 
and let  $(x_i)_{i \in I}$  be a family of points in $X$.
Then there exists a point $x \in X$ such that 
$(x,  x_i) \in U^{(\Omega_i)}$ for all  $i \in I$.
\end{proposition}

\begin{proof}
Denote by $\PPf(G)$ (resp.~$\PPf(I)$) the set of all finite subsets of $G$ (resp.~$I$).
Consider, for $A \in \PPf(G)$ and $J \in \PPf(I)$,  the subset $X(A,J) \subset X$ consisting of all $x \in X$ such that
$(x,x_i) \in U^{(A \cap \Omega_i)}$ for all  $i \in J$.
First observe that since $U$ is closed then $X(A,J)$ is closed in $X$ for all $A \in \PPf(G)$ and $J \in \PPf(I)$.
Indeed,
\[
X(A,J) = \bigcap_{i \in J} \bigcap_{a_i \in A \cap \Omega_i} a_i^{-1} U[a_ix_i].
\]
On the other hand, if we fix $A \in \PPf(G)$ and $J \in \PPf(I)$, then the subsets $A \cap \Omega_i$, $i \in J$,
are finite and satisfy
\[
(A \cap \Omega_j) \cap \Lambda  (A \cap \Omega_k)  = \varnothing \quad \text{for all distinct $j,k \in J$}
\]
by~\eqref{e:lambda-separated}.
Since $\Lambda$ is a specification subset for $(X,G,U)$,
we have  $X(A,J) \neq \varnothing$.
As
\[
X(A_1,J_1) \cap X(A_2,J_2) \cap  \cdots \cap X(A_n,J_n) \supset X(A_1 \cup A_2 \cup  \cdots \cup A_n,J_1 \cup J_2 \cup \dots \cup J_n),
\]
we deduce that  $X(A_1,J_1) \cap X(A_2,J_2)\cap  \dots \cap X(A_n,J_n) \neq \varnothing$ for all
$A_1,A_2, \ldots , A_n \in \PPf(G)$, $J_1,J_2,\ldots,J_n \in \PPf(I)$,  and $n \geq 0$.

Thus $(X(A,J))_{A \in \PPf(G), J \in \PPf(I)}$ is a family of closed subsets of $X$ with the finite intersection property.
By compactness of $X$, the intersection of this family is not empty.
This means that there exists a point $x \in X$ such that $x \in X(A,J)$ for all finite subsets $A \subset G$ and $J \subset I$.
Clearly, such an $x$ satisfies  $(x, x_i) \in U^{(\Omega_i)}$ for all   $i \in I$.
\end{proof}

\begin{proposition}
\label{p:wsp-ultrauniform}
Let $X$ be a ultrauniform space equipped with a uniformly continuous action of a group $G$.
Then the action of $G$ has the weak specification property, that is, satisfies condition (WSP),
if and only if the following condition is satisfied:
\begin{enumerate}[{\rm (WSP')}]
\item for every entourage $U$ of $X$, there exists a finite subset $\Lambda \subset G$ such that the following holds:
if $\Omega_1$ and $\Omega_2$ are two finite subsets of $G$ such that $\Omega_1 \cap \Lambda \Omega_2 = \varnothing$, then,
given any two points $x_1, x_2 \in X$, there exists a point $x \in X$ such that 
$(x, x_k)\in U^{(\Omega_k)}$ for all
 $k \in \{ 1,2\}$.
\end{enumerate}
\end{proposition}

\begin{proof}
As (WSP) trivially implies (WSP'), we only need to prove the converse implication.
Suppose (WSP'). In order to show (WSP), we proceed by induction on $|I|$.
The base case corresponds to $|I| = 2$ and this is indeed our assumption.
Suppose that (WSP) holds whenever $|I| \leq n$ and let $\Lambda = \Lambda(W,n) \subset G$ denote a finite subset
guaranteeing (WSP) for any equivalence entourage $W$ of $X$ and any index set $I$ with $|I| \leq n$.
Fix an entourage $U$ of $X$ and let $I'$ be a finite index set with $|I'| = n+1$.
Since $X$ is ultrauniform, we can find an equivalence entourage $W \subset U$.
Let us show that $\Lambda \coloneqq \Lambda(W,n)$ also satisfies (WSP) for $I'$.
Let $(\Omega_i)_{i \in I'}$ be a family of finite subsets of $G$ such that $\Omega_j \cap \Lambda \Omega_k = \varnothing$
for all distinct $j,k \in I'$, and let $(x_i)_{i \in I'}$ be a family of points in $X$.
Fix $i' \in I'$ and set $I \coloneqq I' \setminus \{i'\}$ so that $|I| = n$.
Then, by the inductive hypothesis, we can find a point $x' \in X$ such that
\begin{equation}
\label{e:x'-xi}
(x', x_i) \in W^{(\Omega_i)} \quad \text{for all $i \in I$}.
\end{equation}
  Set $\Omega_1 \coloneqq \cup_{i \in I} \Omega_i$ and $\Omega_2 \coloneqq \Omega_{i'}$ as
well as $x_1 \coloneqq x'$ and $x_2 \coloneqq x_{i'}$. 
Then $\Omega_1 \cap \Lambda\Omega_2 = \varnothing$ so that we can find $x \in X$ such 
that $(x, x_k) \in W^{(\Omega_k)}$ for all  $k \in \{1,2\}$, 
that is,
\begin{equation}
\label{e:x-x'}
(x, x') \in W^{(\Omega_i)} \quad \text{for all  $i \in I$}
\end{equation}
and
\begin{equation}
\label{e:x-x-i'}
(x, x_{i'}) \in W^{(\Omega_{i'})}.
\end{equation}
Since $W^{(\Omega_i)}$ is an equivalence entourage, 
we deduce from~\eqref{e:x'-xi} and~\eqref{e:x-x'}  that
$(x, x_i) \in W^{(\Omega_i)}$ for all  $i \in I$. 
This, together with~\eqref{e:x-x-i'}, 
yields $(x, x_i) \in W^{(\Omega_i)} \subset U^{(\Omega_i)}$ for all   $i \in I'$.
This completes the inductive argument and shows the implication (WSP') $\implies$ (WSP).
\end{proof}

Let $G$ be a group and let $A$ be a set.
Recall the following definitions \cite[Section~3]{csc-myhill-monatsh}.
Given a finite subset $\Delta \subset G$,
one says
 that a subset  $X \subset A^G$ is
\emph{$\Delta$-irreducible} if it satisfies the following condition:
\begin{enumerate}[{\rm (SI)}]
\item 
if $\Omega_1$ and $\Omega_2$ are two finite subsets of $G$ such that $\Omega_1 \cap \Omega_2 \Delta = \varnothing$, then,
given any two configurations $x_1, x_2 \in X$, there exists a configuration $x \in X$
such that $x\vert_{\Omega_k} = x_k\vert_{\Omega_k}$ for all $k \in \{1,2\}$.
\end{enumerate}
One says that a subset $X \subset A^G$ is
\emph{strongly irreducible} if there exists a finite subset $\Delta \subset G$ such that $X$ is $\Delta$-irreducible.

\begin{proposition}
\label{p:uni-str-irr-vs-wsp}
Let $A$ be a uniform space and let $G$ be a group. 
Equip $A^G$ with the product uniform structure and the shift action of $G$. 
Let $X \subset A^G$ be a strongly irreducible $G$-invariant subset.
Then the uniform dynamical system $(X,G)$  has the weak specification property.
In particular, the full shift $(A^G,G)$ has the weak specification property.
\end{proposition}

\begin{proof}
Let $\Delta \subset G$ be a finite subset such that $X$ is $\Delta$-irreducible.
Fix an entourage $U$ of $X$. Then we can find an entourage $V$ of $A$ and a finite subset $\Omega \subset G$ such that the entourage
\[
W = W(X,V,\Omega) \coloneqq \{(x,y) \in X \times X: (x(g), y(g)) \in V \text{ for all } g \in \Omega\}
\]
is contained in $U$.
We claim that the finite subset $\Lambda  \coloneqq \Omega \Delta^{-1} \Omega^{-1} \subset G$ is a specification subset for $(X,G,U)$.
To see this, let $(\Omega_i)_{i \in I}$ be a finite family of finite subsets of $G$ such that $\Omega_j \cap \Lambda \Omega_k = \varnothing$ for all distinct $j,k \in I$ and let $(x_i)_{i \in I}$ be a family of points in $X$.
Setting $\Omega_i' \coloneqq \Omega_i^{-1}\Omega \subset G$ for all $i \in I$,
we have that $\Omega_j' \cap \Omega_k' \Delta = \varnothing$ for all distinct $j,k \in I$. Since $X$ is 
$\Delta$-irreducible,
using an immediate inductive argument on $|I|$,
we can find a configuration $x \in X$ such that $x\vert_{\Omega_i'} = x_i\vert_{\Omega_i'}$ for all $i \in I$.
This implies that  
$(gx)\vert_{\Omega} = (gx_i)\vert_{\Omega}$ for all $g \in \Omega_i$ and $i \in I$.
It follows that $(x,x_i) \in W^{(\Omega_i)} \subset U^{(\Omega_i)}$ for all $i \in I$.
This proves our claim and shows that $(X,G)$ has the weak specification property.
\end{proof}

The following result extends Proposition~A.1 in~\cite{li-goe}.

\begin{proposition}
\label{p:wsp-strong-irred}
Let $G$ be a group and let $A$ be a set.
Equip $A^G$ with the prodiscrete  uniform structure and the shift action of $G$. 
Let $X \subset A^G$ be a $G$-invariant subset. 
Then the following conditions are equivalent:
\begin{enumerate}[{\rm (a)}]
\item $X$ is strongly irreducible;
\item the $G$-shift action on $X$ has the weak specification property.
\end{enumerate}
\end{proposition}

\begin{proof}
The implication (a) $\implies$ (b) follows from Proposition~\ref{p:uni-str-irr-vs-wsp} after equipping $A$ with the discrete uniform structure.
\par
 Conversely, suppose (b). Let $U = W(1_G) \coloneqq \{(x,y) \in X \times
X: x(1_G) = y(1_G)\}$
and let $\Lambda = \Lambda(U) \subset G$ be a specification subset for $(X,G,U)$.
Let us show that $X$ is $\Delta$-irreducible with $\Delta \coloneqq \Lambda^{-1} \subset G$.
Let $\Omega_1, \Omega_2 \subset G$ such that $\Omega_1 \cap \Omega_2 \Delta = \varnothing$
and let $x_1$ and $x_2$ be two configurations in $X$. As $\Omega_1^{-1} \cap \Lambda \Omega_2^{-1} = \varnothing$ and $\Lambda$ is a  specification subset for $(X,G,U)$, we can find 
a configuration $x \in X$ such that 
$(x, x_k)\in U^{(\Omega_k^{-1})}$ for all
 $k \in \{1,2\}$. 
This implies  $x(g_k) = (g_k^{-1}x)(1_G) = (g_k^{-1}x_k)(1_G) =
x_k(g_k)$ for all $g_k \in \Omega_k$ and $k \in \{1,2\}$. 
It follows that $x\vert_{\Omega_k} = x_k\vert_{\Omega_k}$ for all $k \in \{1,2\}$.
This shows that $X$ is $\Delta$-irreducible. The implication (b) $\implies$ (a) follows.
\end{proof}

The following result extends Corollary~5.2 in ~\cite{li-goe}. 

\begin{theorem}
\label{t:entropie-positive}
Let $X$ be a compact Hausdorff space equipped with a continuous action of an amenable group $G$.
Suppose that $X$ has more than one point  and that the action of $G$ on $X$ has the weak specification property.
Then one has $\htop(X,G) > 0$.
\end{theorem}

\begin{proof}
Let $x_1$ and $x_2$ be two distinct points in $X$. 
Then we can find a symmetric entourage $U$ of $X$ such that
\begin{equation}
\label{e:x-neq-y-et-U}
(x_1,x_2) \notin U \circ U \circ U.
\end{equation}
Let $\Lambda \subset G$ be a specification subset for $(X,G,U)$.
Up to enlarging $\Lambda$ if necessary, we can assume that $\Lambda = \Lambda^{-1}$.
\par
Let $\FF = (F_j)_{j \in J}$ be a left F\o lner net for $G$.
Given $j \in J$, let $\Delta_j \subset F_j$ be a maximal subset subject to the condition
that 
\begin{equation}
\label{e:Delta-j-Lambda-sep}
g \notin  \Lambda h \text{  for all distinct $g,h \in \Delta_j$}.
\end{equation}
By maximality of $\Delta_j$ and the fact that $\Lambda = \Lambda^{-1}$, we have $F_j \subset \Lambda \Delta_j$. This implies
$|F_j| \leq | \Lambda \Delta_j| \leq |\Lambda| \cdot |\Delta_j|$ and hence
\begin{equation}
\label{e:ratio-Delta-j-F_j}
\frac{|\Delta_j|}{| F_j|} \geq \frac{1}{| \Lambda |} \quad \text{for all $j \in J$}.
\end{equation}
\par
As $\Lambda$ is a specification subset for $(X,G,U)$ and $\Delta_j$ satisfies~\eqref{e:Delta-j-Lambda-sep},
for every $x \in \{x_1,x_2\}^{\Delta_j}$,
we can find   $z = z(x) \in X$ such that 
\begin{equation}
\label{e:l'orbite}
(g z, x(g)) = (g z, g(g^{-1} x(g))) \in U \text{  for all $g \in \Delta_j$}.
\end{equation}
Comparing \eqref{e:l'orbite} and \eqref{e:x-neq-y-et-U},
 we deduce that the set $Z_j \coloneqq \{z(x) : x \in \{x_1,x_2\}^{\Delta_j}\} \subset X$ is
$(\Delta_j,U)$-separated and has cardinality
\begin{equation}
\label{e:card-Z-j-bo}
|Z_j| = \vert \{x_1,x_2\}^{\Delta_j} \vert = 2^{|\Delta_j|}.
\end{equation}
 Since $\Delta_j \subset F_j$, the set $Z_j$ is also 
 $(F_j,U)$-separated, so that
\begin{equation}
\label{e:F-j-separated}
\sep(X,G,F_j,U) \geq |Z_j|.
\end{equation}
We conclude that
\begin{align*}
\htop(X,G) & = \hsep(X,G) && \mbox{(by \eqref{e:htop-equals-huni})}\\
& \geq \hsep(X,G,\FF,U) && \text{(by \eqref{e:def-hsep-X-G})} \\  
&= \limsup_{j \in J} \frac{\log \sep(X,G,F_j,U)}{|F_j|}
&& \text{(by~\eqref{d:hsep-U})} \\
& \geq \limsup_{j \in J} \frac{\log |Z_j|}{|F_j|}
&& \text{(by~\eqref{e:F-j-separated})} \\ 
& \geq \limsup_{j \in J} \frac{|\Delta_j| \log 2}{|F_j|} 
&& \text{(by~\eqref{e:card-Z-j-bo})} \\
&\geq \frac{\log 2}{|\Lambda|}  && \text{(by \eqref{e:ratio-Delta-j-F_j})}. 
\end{align*}
This implies $\htop(X,G) > 0$.
\end{proof}

\section{Proof of the main result}
\label{sec:proof-main-result}

The following result extends Proposition~3.1 in~\cite{li-goe}.

\begin{theorem}
\label{t:entropie-sousespace-propre}
Let $X$ be a compact Hausdorff space equipped with a continuous action of an amenable group $G$
and let $Y \subsetneqq X$ be a proper closed $G$-invariant subset of $X$. 
 Suppose that the action of $G$ on $X$ has the weak specification property
 and that the action of $G$ on $Y$ is expansive. 
Then one has  $\htop(Y,G) < \htop(X,G)$.
\end{theorem}

\begin{proof}
To simplify, if $U$ is an entourage of $X$, we shall also write $U$ to denote  the entourage of $Y$  obtained by intersecting $U$  with $Y \times Y$.
\par 
Fix  a point $x_0 \in X \setminus Y$. 
Since  $X$ is compact Hausdorff  and therefore regular, we can find an entourage $W$ of $X$ 
such that $W[x_0]$ does not meet  $Y$, that is, $(x_0,y) \notin W$ for all $y \in Y$.
\par
Let $U_0 \subset X \times X$ be an expansiveness entourage for the action of $G$ on $Y$ and 
take  a closed symmetric entourage $V$ of $X$ such that 
$V \circ V \circ V \circ V \subset  W \cap U_0$.
Choose also  a symmetric entourage $U$ of $X$ such that
$U \circ U \circ U \subset V$. 
\par 
Let  $\Lambda  \subset G$ be a specification subset for $(X,G,U)$.
Up to enlarging $\Lambda$ if necessary, we can assume that $\Lambda = \Lambda^{-1}$ and $1_G \in \Lambda$.
\par
Let $\FF = (F_j)_{j \in J}$ be a left F\o lner net for $G$. 
Given $j \in J$, let $\Delta_j \subset F_j$ be a maximal subset subject to the condition
that $g \notin  \Lambda h$ for all distinct $g,h \in \Delta_j$. 
By maximality of $\Delta_j$ and the fact that $\Lambda = \Lambda^{-1}$ and $1_G \in \Lambda$, we have 
$F_j \subset \Lambda \Delta_j$. This implies
$|F_j| \leq | \Lambda \Delta_j| \leq |\Lambda| \cdot |\Delta_j|$ and hence
\begin{equation}
\label{e:Delta-j-vs-F_j}
\frac{|\Delta_j|}{| F_j|} \geq \frac{1}{| \Lambda |} \quad \text{for all $j \in J$}.
\end{equation}
\par
Fix now $j \in J$ and  $A \subset \Delta_j$. 
Take a minimal   $(F_j \setminus (\Lambda A), V)$-spanning subset $Y_A \subset Y$  
for $Y$, so that
\begin{equation}
\label{e:card-Y-A}
| Y_A| = \spa(Y,G,F_j \setminus (\Lambda A), V).
\end{equation} 
Choose also,  for each $a \in A$, 
a  minimal 
$(\Lambda a, V)$-spanning subset $Z_a \subset Y$ for $Y$, 
so that
\begin{equation}
\label{e:card-Z-a} 
|Z_a| = \spa(Y,G,\Lambda a, V) = \spa(Y,G,\Lambda, V),
\end{equation}
where the last equality follows from~\eqref{e:sep-spa-cov-right-inv}.
\par
Consider  now a point $y \in Y$.
Since $Y_A$ is $(F_j \setminus (\Lambda A), V)$-spanning for $Y$,  
we can find  a point $y_A = y_A(y) \in Y_A$ such that $(g y, g y_A) \in V$ 
for all $g \in F_j \setminus (\Lambda A)$.
On the other hand, since $Z_a$ is $(\Lambda a, V)$-spanning for $Y$,
we can find,  
for each $a \in A$,   a point  $z_a = z_a(y) \in Z_a$ such that $(g y, g z_a) \in V$ for all 
$g \in \Lambda a$.
Now, if $y,y' \in Y$ satisfy  $y_A(y) = y_A(y')$ and $z_a(y) = z_a(y')$ for all $a  \in A$, then 
$(gy, g y') \in V \circ V$ for all $g \in F_j$.
Therefore the map 
\[
y \mapsto \left(y_A(y), (z_a(y))_{a \in A})\right) \in Y_A \times \prod_{a \in A} Z_a
\]  
is injective  on each 
$(F_j, V \circ V)$-separated subset of $Y$.
By taking cardinalities: we deduce that
\begin{align*}
\sep(Y,G,F_j, V \circ V) 
& \leq |Y_A |  \cdot \prod_{a  \in A} | Z_a |\\
& = \spa(Y,G,F_j \setminus (\Lambda A), V) \cdot \spa(Y,G,\Lambda, V)^{|A|}
&&\text{(by \eqref{e:card-Y-A} and \eqref{e:card-Z-a})} \\
&\leq
\sep(Y,G,F_j \setminus (\Lambda A), V) \cdot \sep(Y,G,\Lambda, V)^{|A|}
&& \text{(by~\eqref{e:relations-spa-sep-cov})}.
\end{align*}
This gives us
\begin{equation}
\label{e:sep-V-circV-X-ge-sep-Y}
\sep(Y,G,F_j \setminus (\Lambda A), V) \geq \sep(Y,G,F_j, V \circ V) \cdot \sep(Y,G,\Lambda, V)^{-|A|}.  
\end{equation}
Let now $S_A \subset Y$ be an $(F_j \setminus (\Lambda A), V)$-separated subset of $Y$ with maximal cardinality, so that
\begin{equation}
\label{e:card-S-A-sep}
|S_A| = \sep(Y,G,F_j \setminus (\Lambda A), V).
\end{equation} 
Let $y \in S_A$. 
Take as an index set   $I \coloneqq A \cup \{i_0\}$, where $i_0$ is an index element not in $A$,
and consider the family 
$(\Omega_i)_{i \in I}$ of finite subsets of $F_j$  defined by  
$\Omega_a \coloneqq \{a\}$ for all $a \in A$ and $\Omega_{i_0} \coloneqq F_j \setminus \Lambda A$. 
Observe that $\Omega_i \cap \Lambda \Omega_k = \varnothing$ for all distinct $i,k \in I$. Indeed, $\{a\} \cap \Lambda \{a'\} = \varnothing$
for all distinct $a,a' \in A \subset \Delta_j$, while, obviously, $F_j \setminus (\Lambda A) \cap \Lambda \{a\} = \varnothing$ for all $a \in A$.
Thus, if we consider the family of points $(x_i)_{i \in I}$ in $X$ defined by $x_a \coloneqq a^{-1}x_0$ for all $a \in A$ and
$x_{i_0} \coloneqq y$ 
(the given point in $S_A$), 
condition (WSP) ensures the existence of a point $x = x(y,A) \in X$ such that
$(g_i x,g_i x_i) \in U$ for all $g_i \in \Omega_i$ and $i \in I$, 
that is, $(a x, x_0) = (ax,a(a^{-1}x_0)) \in U$ for all $a \in A$ and
$(g x,g y) \in U$ for all $g \in F_j \setminus (\Lambda A)$.
\par
Now, if $y, y' \in S_A$ are distinct, we can find a group element $g \in F_j \setminus (\Lambda A)$ such that
$(g y,g y') \notin V$.
Setting $x = x(y,A)$ and $x' = x(y',A)$,
this implies $(gx,gx') \notin U$.
Indeed, otherwise,
from $(gy, gx), (gx, gx'), (gx', gy') \in U$ we would deduce
\[
(g y,g y') \in U \circ U \circ U \subset V \mbox{ \ for all } g \in F_j \setminus (\Lambda A),
\]
a contradiction.
\par
Moreover, for $A, B \subset \Delta_j$ distinct, $y \in S_A$ and $y' \in S_B$, and for $c \in A \triangle B$,
where $A \triangle B \coloneqq (A  \setminus  B) \cup (B \setminus  A)$ denotes 
the symmetric difference set of $A$ and $B$, 
we have $(cx(y,A), cx(y',B)) \notin U$. 
Otherwise, if $c =:a \in A \setminus B$ (so that $a \in F_j \setminus \Lambda B$)
from $(x_0, ax), (ax, ax'), (ax', ay')  \in U$, we would deduce
\[
(x_0, ay') \in U \circ U \circ U   \subset V  \subset W,
\]
a contradiction since $ay' \in Y$.
Similarly, assuming  $c =:b \in B \setminus A$ 
we would get $(x_0, by) \in W$, again a contradiction since $by \in Y$.
\par
It follows that the set $\{x(y,A) :\;  y \in S_A, A \subset \Delta_j\} \subset X$ is 
$(F_j, U)$-separated with cardinality
$\sum_{A \subset \Delta_j} |S_A| = \sum_{A \subset \Delta_j} \sep(Y,G,F_j\setminus (\Lambda A), V)$. 
We deduce that
\begin{align*}
\label{e:h-X-vs-h-Y}
\sep(X,G,F_j, U) & \geq \sum_{A \subset \Delta_j} \sep(Y,G,F_j\setminus (\Lambda A), V)\\
& \geq \sum_{A \subset \Delta_j} \sep(Y,G,F_j, V \circ V) \cdot \sep(Y,G,\Lambda, V)^{-|A |}
&& \text{(by \eqref{e:sep-V-circV-X-ge-sep-Y})} \\
&= \sep(Y,G,F_j, V \circ V) \cdot \sum_{A \subset \Delta_j} \sep(Y,G,\Lambda, V)^{-|A |} \\
&= \sep(Y,G,F_j, V \circ V) \left(1 + \sep(Y,G,\Lambda, V)^{-1} \right)^{|\Delta_j|}. 
\end{align*}
This gives us  
\begin{align*}
\frac{\log \sep(X,G,F_j, U)}{|F_j|} 
&\geq 
\frac{\log \sep(Y,G,F_j, V \circ V)}{|F_j|} 
+ \frac{|\Delta_j|}{|F_j|} \log \left(1 + \sep(Y,G,\Lambda, V)^{-1}\right) \\ 
& \geq 
\frac{\log \sep(Y,G,F_j, V \circ V)}{|F_j|} 
+ \frac{1}{| \Lambda|} \log \left(1 + \sep(Y,G,\Lambda, V)^{-1}\right), 
\end{align*}
where the second inequality follows from~\eqref{e:Delta-j-vs-F_j}.
 \par
Finally, we obtain 
\begin{align*}
\htop(X,G) & = \hsep(X,G) 
&& \text{(by~\eqref{e:htop-equals-huni})} \\ 
 &\geq \hsep(X,G,\FF, U) 
 &&\text{(by~\eqref{e:def-hsep-X-G})} \\ 
&=   \limsup_{j \in J} \frac{\log \sep(X,G,F_j, U)}{|F_j|} \\
& \geq  \limsup_{j \in J} \frac{\log \sep(Y,G,F_j, V \circ V)}{|F_j|} +  
\frac{1}{|\Lambda|} \log \left(1 + \sep(Y,G,\Lambda, V)^{-1}\right) \\
& =
\hsep(Y,G,\FF,V \circ V) 
+ \frac{1}{|\Lambda|} \log \left(1 + \sep(Y,G,\Lambda, V)^{-1} \right) \\
&\geq
\hcov(Y,G,\FF,V \circ V \circ V \circ V) 
+ \frac{1}{|\Lambda|} \log \left(1 + \sep(Y,G,\Lambda, V)^{-1} \right)
&& \text{(by~\eqref{e:relations-spa-sep-cov-h})} \\
 & =
\htop(Y,G) 
+ \frac{1}{|\Lambda|} \log \left(1 + \sep(Y,G,\Lambda, V)^{-1} \right),
\end{align*}
where the last equality
 follows from Theorem~\ref{t:entropy-expansive},
since $V \circ V \circ V \circ V \subset U_0$
is a closed expansiveness entourage for $(Y,G)$. 
 This shows that $\htop(X,G) > \htop(Y,G)$.  
\end{proof} 

The following result extends Proposition~3.2 in~\cite{li-goe}.

\begin{theorem}
\label{t:entropie-diminue-failure-preinjectivite}
Let $X$ and $Y$ be  compact Hausdorff spaces 
equipped with   expansive continuous actions of an amenable group   $G$. 
Suppose that the action of $G$ on $X$ has the weak specification property
and that $\htop(Y,G) < \htop(X,G)$.  
Let   $f \colon X \to Y$ be a 
continuous $G$-equivariant map.
Then $f$ is not pre-injective.
In fact, 
the restriction of $f$ to any homoclinicity class of $X$ fails to be  injective.
\end{theorem}
  
\begin{proof}
Let $\FF = (F_j)_{j \in J}$ be a left F\o lner net for $G$.
Let $U_0$ (resp.~$V_0$) be an 
 expansiveness entourage for $(X,G)$ (resp.~$(Y,G)$). 
Choose    a closed symmetric entourage $V$ of $Y$ such that $V \circ V \circ V \circ V  \subset V_0$.
Since $f$ is uniformly continuous, 
we can find a closed symmetric entourage $U$ of $X$ such that
$U \circ U \circ U \circ U \subset U_0$ and
\begin{equation}
\label{e:f-unif-cont-U-V}
(f \times f)(U) \subset V.
\end{equation}
Fix a point  $x \in X$.
Let $\Lambda \subset G$ be a specification subset for $(X,G,U)$. It is not restrictive to suppose that $\Lambda = \Lambda^{-1}$.
\par
Since  $U \circ U$ is a closed symmetric entourage of $X$ and
$U \circ U \circ U \circ U$ is an expansiveness entourage for $(X,G)$,  we have  
\begin{align*}
\htop(X,G)
&= \hsep(X,G,\FF,U \circ U)
&& \text{(by Corollary~\ref{c:sep-spa-expansiveness})} \\
&= \limsup_{j \in J}\frac{\log \sep(X,G,F_j,U \circ U)}{|F_j|}
&&\text{(by~\eqref{d:hsep-U})}.
\end{align*}
On the other hand, We have
\begin{align*}
\htop(Y,G)
&= \hspa(Y,G)
&& \text{(by~ Theorem~\ref{t:top-ent-uni})} \\
 & \geq  \hspa(Y,G,\FF,V)
 && \text{(by~\eqref{d:hspa})} \\
 & = \limsup_{j \in J}\frac{\log \spa(Y,G,F_j,V)}{|F_j|}
  &&\text{(by~\eqref{d:hspa-U})}.
\end{align*}
Fixing some constant     $\eta > 0$ such that $3 \eta < \htop(X,G) - \htop(Y,G)$,
we deduce that 
for all  $j_0 \in J$, there exists $j \in J$ with $j \geq j_0$ such that
\[
\frac{\log \sep(X,G,F_j,U \circ U)}{|F_j|} \geq \htop(X,G) - \eta,
\text{  and  }
\frac{\log \spa(Y,G,F_j,V)}{|F_j|} \leq \htop(X,G) - 2\eta,
\]
so that
\begin{equation}
\label{e:sep-exp-spa}
\sep(X,G,F_j,U \circ U) \geq \spa(Y,G,F_j,V) \cdot \exp(\eta|F_j|). 
\end{equation}
By virtue of Assertion~(i) in Lemma~\ref{l:finite-times-folner},
 we can furthermore assume  that
\[
\frac{|\Lambda F_j \setminus F_j|}{|F_j|} <  \frac{\eta}{\log \cov(Y,G,\{1_G\},V)},
\]
so that
\begin{equation}
\label{e:cov-V-Y-F-j-eta}
(\cov(Y,G,\{1_G\},V))^{|\Lambda F_j \setminus F_j|} < \exp(\eta|F_j|).
\end{equation}
\par
Let now $Z_j \subset X$ be an $(F_j,U \circ U)$-separated subset with maximal cardinality, so that
\begin{equation}
\label{e:card-Z-j-sep}
|Z_j| = \sep(X,G,F_j,U \circ U) ,
\end{equation}
and let $S_j \subset Y$ be an $(F_j,V)$-spanning subset with minimal cardinality, so that
\begin{equation}
\label{e:card-S-j-spa}
|S_j| = \spa(Y,G,F_j,V).
\end{equation}
The fact that  $S_j$ is $(F_j,V)$-spanning, 
implies that, for each  $z \in Z_j$, we can find $s \in S_j$ such that
$(f(z),s) \in V^{(F_j)}$.
By~\eqref{e:sep-exp-spa}, \eqref{e:card-Z-j-sep}, and \eqref{e:card-S-j-spa},  
we have
\[
|Z_j| \geq  |S_j| \cdot \exp(\eta|F_j|).
\]
Consequently,
it follows  from  the pigeon-hole principle that 
there exists  a subset $\widetilde{Z}_j \subset Z_j$ such that
\begin{align}
\label{e:lower-Z-tilde-j}
 |\widetilde{Z}_j| \geq \exp(\eta |F_j|)
 \end{align}
  and an element
$s_0 \in S_j$ such that
 all $z \in \widetilde{Z}_j$ satisfy
\begin{equation}
\label{e:f-z-s-0-Fj-V}
(f(z),s_0) \in V^{(F_j)}.
\end{equation}
\par
Since $\Lambda$ is a specification subset for $(X,G,U)$ such that $\Lambda = \Lambda^{-1}$, so that
$F_j \cap \Lambda(G \setminus \Lambda F_j) = \varnothing$,
it follows from Proposition~\ref{p:wspec-infinite} that
for each $z \in \widetilde{Z}_j$, we can find $z' \in X$ satisfying 
\begin{equation}
\label{e:z-z'-F-j}
(z',z) \in U^{(F_j)}
\end{equation}
and
\begin{equation}
\label{e:z'-x-out-Lambda-F-j}
(z',x) \in U^{(G \setminus \Lambda F_j)}.
\end{equation}
As $U$ is a closed expansiveness entourage for $(X,G)$, 
we deduce from  Proposition~\ref{p:homoclinic-equiv} and~\eqref{e:z'-x-out-Lambda-F-j}  
that  $z'$ is homoclinic to  $x$.
\par 
On the other hand,
it follows from~\eqref{e:z-z'-F-j}, \eqref{e:f-unif-cont-U-V}, and the $G$-equivariance of $f$ that
\begin{equation}
\label{e:f-z-z'-F-j}
(f(z'),f(z)) \in V^{(F_j)}.
\end{equation}
Similarly, \eqref{e:z'-x-out-Lambda-F-j}, \eqref{e:f-unif-cont-U-V}, and the $G$-equivariance of $f$
imply that
\begin{equation}
\label{e:f-z'-x-out-Lambda-F-j}
(f(z'),f(x)) \in V^{(G \setminus \Lambda F_j)}.
\end{equation}
\par
Now, as 
\[
|\widetilde{Z}_j| > (\cov(Y,G,\{1_G\},V))^{|\Lambda F_j \setminus F_j|}
\]
by~\eqref{e:lower-Z-tilde-j} and~\eqref{e:cov-V-Y-F-j-eta},
it follows again from the pigeon-hole principle that  we can find two distinct points 
$z_1,z_2 \in \widetilde{Z}_j$ such that
\begin{equation}
\label{e:f-z1-z2-Lambda-Fj-minus-Fj}
(f(z_1'),f(z_2')) \in V^{(\Lambda F_j \setminus F_j)}
\subset (V \circ V \circ V \circ V)^{(\Lambda F_j \setminus F_j)}.
\end{equation}
By~\eqref{e:f-z-z'-F-j} and~\eqref{e:f-z-s-0-Fj-V}, 
we  have
\begin{equation}
\label{e:f-z1'-z2'-V-Fj}
(f(z_1'),f(z_2')) \in (V \circ V \circ V \circ V)^{(F_j)}.
\end{equation}
On the other hand, by using~\eqref{e:f-z'-x-out-Lambda-F-j}, 
we get
\begin{equation}
\label{e:f-z1'-z2'-G-setminus-Fj}
(f(z_1'),f(z_2')) \in (V \circ V)^{(G \setminus \Lambda F_j)}
\subset (V \circ V \circ V \circ V)^{(G \setminus \Lambda F_j)}.
\end{equation}
Combining~\eqref{e:f-z1-z2-Lambda-Fj-minus-Fj}, \eqref{e:f-z1'-z2'-V-Fj},
and~\eqref{e:f-z1'-z2'-G-setminus-Fj}, 
we obtain
\[
(f(z_1'),f(z_2')) \in (V \circ V \circ V \circ V)^{(G)},
\]
which implies $f(z_1') = f(z_2')$ since $V \circ V \circ V \circ V$ is an expansiveness entourage 
for $(Y,G)$.
As $z_1'$ and $z_2'$ are in the homoclinicity class of $x$,
it remains only to show that the points $z_1'$ and $z_2'$ are distinct.
But this is clear since otherwise
\eqref{e:z-z'-F-j}
would then imply $(z_1,z_2) \in (U \circ U)^{(F_j)}$, a contradiction since  $z_1$ and $z_2$ are distinct points in $\widetilde{Z}_j \subset Z_j$ and
$Z_j$ is  $(F_j,U \circ U)$-separated.
\end{proof}

Combining the two previous theorems, we get the following.

\begin{theorem}
\label{t:pre-inj-same-htop}
Let $X$ and $Y$ be compact Hausdorff spaces 
equipped with expansive continuous actions of an amenable group $G$. 
Suppose that the actions of $G$ on $X$ and $Y$ have the weak specification property
and that $\htop(X,G) =\htop(Y,G)$.  
Then every pre-injective continuous $G$-equivariant map
$f \colon X \to Y$ is surjective.
\qed
\end{theorem}

\begin{proof}[Proof of Theorem~\ref{t:myhill}]
This is  Theorem~\ref{t:pre-inj-same-htop} with $X = Y$.
\end{proof}

When $G$ is a group, $A$ and $B$ are finite sets,  $X \subset A^G$ and $Y \subset B^G$ are subshifts, a continuous $G$-equivariant map 
$f \colon X \to Y$ is also  called a \emph{cellular automaton} 
(this terminology is widely used among computer scientists, see e.g.~\cite{book}).

\begin{corollary}
Let $G$ be an amenable  group,  $A$ and $B$ be  finite sets,
and  $X \subset A^G$ and $Y \subset B^G$ be subshifts.
Suppose that $X$ and $Y$ are strongly irreducible and $\htop(X,G) = \htop(Y,G)$.
Then every pre-injective cellular automaton $f \colon X \to Y$ is surjective. 
\end{corollary}

\begin{proof}
This is an immediate consequence of Theorem~\ref{t:pre-inj-same-htop}
since $(X,G)$ and $(Y,G)$ are expansive (cf.~Example~\ref{ex:subshifts-expansive})
and have the weak specification property by Proposition~\ref{p:wsp-strong-irred}.
\end{proof}

\bibliographystyle{siam}
\def\cprime{$'$}

\end{document}